\documentclass[12pt,reqno,tbtags]{amsart}
\pdfoutput=1
\usepackage{amssymb}
\usepackage{txfonts}
\usepackage[colorlinks=true, pdfstartview=FitV, linkcolor=blue,
  citecolor=blue, urlcolor=blue]{hyperref}
\usepackage[author-year,initials]{amsrefs}
\usepackage{enumerate}

\numberwithin{equation}{section}

\newtheorem{maintheorem}{Theorem}
\newtheorem{conjecture}{Conjecture}
\newtheorem{theorem}{Theorem}[section]
\newtheorem*{theorem*}{Theorem}
\newtheorem{lemma}[theorem]{Lemma}
\newtheorem{proposition}[theorem]{Proposition}

\theoremstyle{definition}{

}

\theoremstyle{remark}{
\newtheorem{remark}{Remark}
\newtheorem*{remark*}{Remark}

\newtheorem{rmkauth}{\textbf{Remark To DAL}}
}

\newenvironment{enumeratei}{\begin{enumerate}[\upshape (i)]}
                           {\end{enumerate}}

\newcommand\floor[1]{\lfloor#1\rfloor}

\newcommand{\R}{\mathbb R}

\newcommand{\Z}{\mathbb Z}

\newcommand{\E}{\mathbf{E}}
\renewcommand{\P}{\mathbf{P}}
\DeclareMathOperator{\var}{Var}

\newcommand{\ep}{\varepsilon}

\renewcommand\phi{\varphi}

\newcommand{\tS}{\tilde{S}}
\newcommand{\dist}{{\rm dist}}

\newcommand{\given}{\, \big| \,}

\newcommand{\tsigma}{\tilde{\sigma}}
\newcommand{\tX}{\tilde{X}}

\newcommand{\spin}{{\mathcal S}}
\newcommand{\taumag}{\tau_{{\rm mag}}}
\newcommand{\cn}{\gamma}
\newcommand{\tabs}{\tau_{\textrm{abs}}}
\newcommand{\one}{\boldsymbol{1}}

\newcommand{\st}{\,:\,}
\newcommand{\ed}{\mathcal{E}}
\newcommand{\deq}{:=}
\newcommand{\Pm}{P}
\newcommand{\X}{\Omega}
\newcommand{\tmix}{t_{{\rm mix}}}

\newcommand{\F}{\mathcal{F}}

\newcommand{\splus}{\xi^+}

\begin{document}
\title[Glauber Dynamics in Mean-Field]{Glauber
  dynamics for the Mean-field Ising Model:
  cut-off, critical power law, and metastability}

\date{\today}

\author{David A. Levin}
\address{Department of Mathematics, University of Oregon,
	Eugene, Oregon 97402-1222}
\email{dlevin@uoregon.edu}
\urladdr{http://www.uoregon.edu/\~{}dlevin}

\author{Malwina J. Luczak}
\address{Department of Mathematics, London School of Economics,
  Houghton Street, London WC2A 2AE, United Kingdom}
\email{m.j.luczak@lse.ac.uk}
\urladdr{http://www.lse.ac.uk/people/m.j.luczak@lse.ac.uk/}

\author{Yuval Peres}
\address{Microsoft Research and University of California, Berkeley}
\email{peres@microsoft.com}
\urladdr{}

\thanks{Research of Y.\ Peres was supported in part by NSF grant DMS-0605166}
\begin{abstract}
  We study the Glauber dynamics for the
  Ising model on the complete graph, also known
  as the Curie-Weiss Model.  For
  $\beta < 1$, we prove that the dynamics exhibits a
  cut-off: the distance to stationarity drops from
  near $1$ to near $0$ in a window of
  order $n$ centered at $[2(1-\beta)]^{-1} n\log n$.
  For $\beta = 1$, we prove that the mixing time is
  of order $n^{3/2}$.  For $\beta > 1$, we study metastability.
  In particular, we show that the Glauber dynamics restricted
  to states of non-negative magnetization has mixing time 
  $O(n \log n)$.
 \end{abstract}

\maketitle

\section{Introduction}
\label{sec.intro}

\subsection{Ising model and Glauber dynamics}
\label{subs.ising}

Let $G = (V, {\mathcal E})$ be a finite graph.  Elements of the state
space $\X \deq \{-1,1\}^V$ will be called \emph{configurations}, and 
for $\sigma \in \X$, the value $\sigma(v)$ will be called the 
\emph{spin} at $v$.  The \emph{nearest-neighbor energy} $H(\sigma)$ of
a configuration $\sigma \in \{-1,1\}^V$ is defined by
\begin{equation} \label{eq.energy}
  H(\sigma) 
    \deq -\sum_{\substack{v,w \in V,\\ v \sim w}} 
      J(v,w) \sigma(v)\sigma(w),   
\end{equation}
where $w \sim v$ means that $\{w,v\} \in \ed$.  The parameters
$J(v,w)$ measure the interaction strength between vertices;
we will always take $J(v,w) \equiv J$, where $J$ is a positive
constant.

For $\beta \geq 0$, the \emph{Ising model} on the graph $G$ with
parameter $\beta$ is the probability measure $\mu$
on $\X$ given by 
\begin{equation} \label{eq.gibbs}
  \mu(\sigma) = \frac{e^{-\beta H(\sigma)}}{Z(\beta)},
\end{equation}
where $Z(\beta) = \sum_{\sigma \in \X} e^{-\beta H(\sigma)}$
is a normalizing constant.

The parameter $\beta$ is interpreted physically as the inverse 
of temperature, and measures the influence of the energy
function $H$ on the probability distribution.  
At \emph{infinite temperature}, corresponding to $\beta = 0$,
the measure $\mu$ is uniform over $\X$ and the random
variables $\{ \sigma(v) \}_{v \in V}$ are independent. 

The (single-site) \emph{Glauber dynamics} for $\mu$ is the Markov
chain on $\X$ with transitions as follows: When at $\sigma$, a
vertex $v$ is chosen uniformly at random from $V$, and a new
configuration is generated from $\mu$ conditioned on the set  
\begin{equation*}
  \{ \eta \in \X \st \eta(w) = \sigma(v), \; w \neq v \}.
\end{equation*}
In other words, if vertex $v$ is selected, the new configuration will 
agree with $\sigma$ everywhere except possibly at $v$, and at
$v$ the spin is $+1$ with probability 
\begin{equation} \label{eq.glauber}
  p(\sigma ; v) 
    \deq \frac{ e^{\beta S^v(\sigma)} }{ 
                e^{\beta S^v(\sigma)} 
		  + e^{-\beta S^v(\sigma)} },
\end{equation}
where $S^v(\sigma) \deq J\sum_{w \,:\, w\sim v} \sigma(w)$.  
Evidently, the distribution of the new spin at $v$ depends only on the 
current spins at the neighbors of $v$.  It is easily seen
that $(X_t)$ is reversible with respect to the measure
$\mu$ in \eqref{eq.gibbs}.

In what follows, the Glauber dynamics will be denoted by
$(X_t)_{t=0}^\infty$.  We use $\P_\sigma$ and $\E_{\sigma}$
respectively to denote the underlying probability measure and
associated expectation operator when $X_0=\sigma$. 

A \emph{coupling} of the Glauber dynamics with starting
states $\sigma$ and $\tilde{\sigma}$ is a process 
$(X_t, \tilde{X}_t)_{t \geq 0}$ such that
$(X_t)$ is a version of the Glauber dynamics with
starting state $\sigma$ and $(\tilde{X}_t)$ is a version
of the Glauber dynamics with starting state $\tilde{\sigma}$.
If a coupling $(X_t, \tilde{X}_t)$ is a Markov chain, we
call it a \emph{Markovian coupling}.
We write $\P_{\sigma,\tsigma}$ and $\E_{\sigma,\tsigma}$ for
the probability measure and associated expectation respectively
corresponding  to a coupling with initial states
$\sigma$ and $\tilde{\sigma}$.

\subsection{Order $n\log n$ mixing and cut-off}

Given a sequence  $G_n = (V_n, E_n)$ of graphs, we write $\mu_n$ for
the Ising measure and $(X^n_t)$ for the
Glauber dynamics on $G_n$. 
The worst-case distance to stationarity of the Glauber
dynamics chain after $t$ steps is  
\begin{equation} \label{eq.disttostat}
  d_n(t) 
    \deq \max_{\sigma \in \X_n} \| \P_\sigma( X^n_t \in \cdot)
      - \mu_n \|_{{\rm TV}} ,
\end{equation}
where $\| \mu - \nu \|_{{\rm TV}}$ denotes the total variation
distance between the probability measures $\mu$ and $\nu$.
The \emph{mixing time}
$\tmix(n)$ is defined as 
\begin{equation} \label{eq.tmix-defn}
  \tmix(n) \deq \min\{t \,:\, d_n(t) \leq 1/4 \}.
\end{equation}
Note that $\tmix (n)$ is finite for each fixed $n$ since, 
by the convergence theorem for ergodic Markov chains,
$d_n(t) \rightarrow 0$ as $t \rightarrow \infty$.
Nevertheless, $\tmix (n)$ will in general tend to infinity with $n$.
Our concern here is with the growth rate of the sequence
$\tmix(n)$.

The Glauber dynamics is said to exhibit a \emph{cut-off}
at $\{t_n\}$ with \emph{window} $\{w_n\}$ if $w_n = o(t_n)$ and
\begin{align*}
  \lim_{\cn \rightarrow \infty} \liminf_{n \rightarrow \infty} 
    d_n(t_n - \cn w_n) & = 1, \\
  \lim_{\cn \rightarrow \infty} \limsup_{n \rightarrow \infty}
    d_n(t_n + \cn w_n) & = 0.
\end{align*}

The first part of this paper is motivated by the following conjecture,
due to the third author:
\begin{conjecture}
  Let $(G_n)$ be a sequence of transitive graphs.
  If the Glauber dynamics on $G_n$ has $\tmix (n) = O(n\log n)$, then
  there is a cut-off.
\end{conjecture}

We establish this conjecture in the special case when $G_n$ is the
complete graph on $n$ vertices and $\beta < 1$ (the ``high
temperature'' regime), where the Glauber dynamics has $O(n \log n)$
mixing time.  

\subsection{Results}

Here we take $G_n$ to be $K_n$, the complete graph on $n$ vertices.
That is, the vertex set is $V_n = \{1,2,\ldots,n\}$, and the edge set
${\mathcal E}_n$ contains all $\binom{n}{2}$ pairs $\{i, j\}$ for 
$1 \leq i < j \leq n$.  We take the interaction parameter
$J$ to be $1/n$; in this case, the Ising measure
$\mu$ on $\{-1,1\}^n$ is given by
\begin{equation} \label{eq.gibbs-cw}
  \mu(\sigma) 
    = \mu_n (\sigma) 
    = \frac{1}{Z(\beta)} 
      \exp\left( \frac{\beta}{n}\sum_{1 \leq i < j \leq n} 
        \sigma(i)\sigma(j) \right).
\end{equation}
In the physics literature, this is usually referred to as the
\emph{Curie-Weiss} model.  For the remainder of this paper,
\emph{Ising model} will always refer to the measure $\mu$ in 
\eqref{eq.gibbs-cw}, and \emph{Glauber dynamics} will always
refer to the one corresponding to this measure.  We will often omit
the explicit dependence on $n$ in our notation.

It is a consequence of the Dobrushin-Shlosman uniqueness criterion
that $\tmix(n) = O(n\log n)$ when $\beta < 1$ \cite{AH}.
See also Bubley and Dyer \ycite{BD:PC}.   Our first result
is that there is a cut-off phenomenon in this regime:

\begin{maintheorem} \label{thm.hightemp}
  Suppose that $\beta < 1$.  The Glauber dynamics for the Ising
  model on $K_n$ has a cut-off at $t_n = [2(1-\beta)]^{-1}n\log n$ with
  window size $n$. 
\end{maintheorem}

\begin{remark}
  Most examples of Markov chains for which the cut-off phenomenon
  has been proved tend to have ample symmetry, for example, random walks
  on groups.    Part of the interest in Theorem \ref{thm.hightemp}
  is that the chain studied here is not of this type, and
  our methods are strictly probabilistic -- in particular,
  based on coupling.  Recently, Diaconis and Saloff-Coste
  \ycite{DSC} gave a sharp criterion for cut-off 
  (for separation distance) for birth-and-death chains.
\end{remark}
  
In the critical case $\beta = 1$, we prove that the mixing time
of the Glauber dynamics is order $n^{3/2}$.

\begin{maintheorem} \label{thm.critical}
  If $\beta = 1$, then there are constants $C_1,C_2 >0$ such that
  for the Glauber dynamics for the Ising model on $K_n$,
  \begin{equation*}
    C_1n^{3/2} \leq \tmix(n) \leq C_2 n^{3/2}.
  \end{equation*}
\end{maintheorem}

Finally, we consider the low-temperature case corresponding
to $\beta > 1$.   To state our result, it is necessary to
mention here the \emph{normalized magnetization}, 
the function $S$ defined on configurations $\sigma$ by
$S(\sigma) \deq n^{-1}\sum_{i=1}^n \sigma(i)$.  
Also, we define the set $\Omega^+$ of states with non-negative 
magnetization,
\begin{equation*}
  \Omega^+ \deq \{ \omega \in X \st S(\sigma) \geq 0 \}.
\end{equation*}
By using the Cheeger inequality with estimates on
the stationary distribution of the magnetization,
the mixing time is seen to be at least exponential in $n$ -- slow mixing
indeed.   Arguments for exponentially slow mixing in the high
temperature regime go back at least to Griffiths, Weng and 
Langer \ycite{GWL}.  

In contrast, we prove that the mixing time is of the order $n \log n$
if the chain is restricted to the set $\Omega^+$.
To be precise, the restricted dynamics evolve as follows on 
$\Omega^+$:  Generate a candidate move $\eta$ according to the usual
Glauber dynamics.  If $S(\eta) \geq 0$, accept $\eta$ as the new state,
while if $S(\eta) < 0$, move instead to $-\eta$.

\begin{maintheorem} \label{thm.lowtemp}
  If $\beta > 1$ then there exist constants $C_3(\beta),C_4(\beta) >0$
  depending on $\beta$ such that,
  for the restricted Glauber dynamics for the Ising model on $K_n$,
  \begin{equation*}
    C_3(\beta) n\log n \leq \tmix(n) \leq C_4(\beta) n\log n.
  \end{equation*}
\end{maintheorem}

For other work on the metastability of related models, see
Bovier, Eckhoff, Gayrard, and Klein \ycites{BEGKa,BEGK},
and Bovier and Manzo \ycite{BM:MetaStabGlaub}. 

The rest of the paper is organized as follows: 
Section \ref{sec.prelim} contains some preliminary lemmas
required in our proofs. Theorems \ref{thm.hightemp}, \ref{thm.critical}
and \ref{thm.lowtemp} are proved in Sections
\ref{sec.cutoff}, \ref{sec.critical}, and
\ref{sec.lowtemp}, respectively.  Section \ref{sec.conj}
contains some conjectures and open problems.

\section{Preliminaries} \label{sec.prelim}

\subsection{Glauber dynamics for Ising on $K_n$}
We introduce here some notation specific to
our setting of the Glauber dynamics for the Ising model
on $K_n$.  For a configuration $\sigma$, recall that
the normalized magnetization $S(\sigma)$ is defined as
\begin{equation*}
  S(\sigma) \deq \frac{1}{n} \sum_{j=1}^n \sigma(j) .
\end{equation*}
Given that the current state of the chain is $\sigma$
and  a site $i$ has been selected for updating,
the probability $p(\sigma, i)$ of updating to a positive
spin, displayed in \eqref{eq.glauber}, is in this case 
$p_+(S(\sigma) - n^{-1}\sigma(i))$, where $p_+$ is the function
given by
\begin{subequations} \label{eq.glauber-cw}
\begin{equation} 
  p_+(s) 
    \deq \frac{e^{\beta s}}{e^{\beta s} + e^{-\beta s}} 
    = \frac{1 + \tanh(\beta s )}{2}.
\end{equation}
Similarly, the probability of updating site $i$ to a negative
spin is $p_-(S(\sigma) - n^{-1}\sigma(i))$, where
\begin{equation} 
  p_-(s) 
    \deq \frac{e^{-\beta s}}{e^{\beta s} + e^{-\beta s} } 
    = \frac{1 - \tanh(\beta s)}{2}.
\end{equation}
\end{subequations}

\subsection{Monotone coupling} \label{sec.mon-coupling}

We  now describe a process called the \emph{grand coupling},
a Markov chain $(\{X_t^\sigma\}_{\sigma \in \Omega})_{t \geq 0}$ 
such that for each $\sigma \in \Omega$, the coordinate process
$(X_t^\sigma)_{t \geq 0}$ is a version of the Glauber
dynamics started at $\sigma$.   It will suffice to describe
one step of the dynamics.  Let $I$ be drawn uniformly
from the sites $\{1,2,\ldots,n\}$, and let $U$ be a
uniform random variable on $[0,1]$, independent of $I$.  For each
$\sigma \in \X$, let $U$ determine the spin $\spin^\sigma$ according
to 
\begin{equation*}
  \spin^\sigma
  =
  \begin{cases}
    +1 & 0 < U \leq p_+(S(\sigma) - n^{-1}\sigma(I))),\\
    -1 & p_+(S(\sigma) - n^{-1}\sigma(I)) < U \leq 1.
  \end{cases} 
\end{equation*}
For each $\sigma$, generate the next state $X_1^\sigma$
according to
\begin{equation*}
  X_1^\sigma(i) =
  \begin{cases}
    \sigma(i) & i \neq I \\
    \spin^\sigma  & i = I
  \end{cases}.
\end{equation*}
We write $\P_{\vec{\sigma}}$ and $\E_{\vec{\sigma}}$ for
the probability measure and expectation operator 
on the measure space where the grand coupling is defined.

For a given pair of configurations, $\sigma$ and
$\tilde{\sigma}$, the two-dimensional projection
of the grand coupling, $(X^\sigma_t, X^{\tilde{\sigma}}_t)_{t \geq 0}$,
will be called the \emph{monotone coupling} with starting
states $\sigma$ and $\tilde{\sigma}$.

For two configurations $\sigma$ and $\sigma'$, the \emph{Hamming
distance} between $\sigma$ and $\sigma'$ is the number of sites
where the two configurations disagree, that is
\begin{equation} \label{eq.HammingMetric}
  \dist(\sigma, \sigma') 
    \deq \frac{1}{2} \sum_{i=1}^n |\sigma(i) - \sigma'(i)| .
\end{equation}

\begin{proposition} \label{prop.hamcontr}
  The monotone coupling $(X_t, \tilde{X}_t)$ of the
  Glauber dynamics  started from $\sigma$ and $\tilde{\sigma}$
  satisfies
  \begin{equation} \label{eq.hamcont}
    \E_{\vec{\sigma}}\left[\dist(X_t, \tilde{X}_t)\right]
      \leq \rho^t \dist(\sigma, \tilde{\sigma}) ,
  \end{equation}
  where
  \begin{equation} \label{eq.rhodefn}
    \rho \deq 1 - n^{-1}\left( 1 - n \tanh(\beta/n)\right).
  \end{equation}
\end{proposition}
\begin{proof}
  We first show that \eqref{eq.hamcont} holds with $t=1$
  provided $\dist(\sigma, \tilde{\sigma}) = 1$.  Indeed,
  suppose that $\sigma$ and $\tsigma$ agree everywhere except at $i$, 
  where $\sigma(i) = -1$ and $\tsigma(i) = +1$.

  Recall that the vertex which is updated in all configurations
  in the grand coupling is denoted by $I$.
  If $I = i$, then the distance decreases by $1$; 
  if $I \neq i$ and the event $B(I)$ occurs, where 
  \begin{equation*}
    B(j) \deq \left\{p_+(S(\sigma)-\sigma(j)/n) \leq U 
       \leq p_+(S(\tsigma)-\tsigma(j)/n) \right\},
  \end{equation*}
  then the distance increases by $1$.
  In all other cases, the distance remains the same.
  Consequently,
  \begin{equation} \label{eq.sdiff}
    \dist(X_1, \tilde{X}_1) 
    = 1 - \one\{I = i\} 
     + \sum_{j \neq i} \one\{I = j\} \one_{B(j)}  .
  \end{equation}
  Note that
  $S(\tsigma) - \tsigma(j)/n = S(\sigma) - \sigma(j)/n + 2/n$ 
  for $j \neq i$.  Thus, letting 
  $\hat{s}_j/n = S(\sigma) - \sigma(j)/n$, for $j \neq i$,
  \begin{equation} \label{eq.pbub}
    \P_{\vec{\sigma}}( B(j) )
      = \frac12 \left[\tanh(\beta (\hat{s}_j+2)/n)) 
       - \tanh(\beta \hat{s}_j/n) \right] \\
     \leq \tanh(\beta/n) .
  \end{equation}
  Taking expectation in \eqref{eq.sdiff}, 
  by the independence of $U$ and $I$ together with \eqref{eq.pbub},
  \begin{equation}
    \E_{\vec{\sigma}}[ \dist(X_1, \tX_1) ]
      \leq 1 - \frac{1}{n} + \tanh(\beta/n) 
      = \rho  
  \end{equation}
  This establishes \eqref{eq.hamcont} for the case
  where $\sigma$ and $\sigma'$ are at unit distance.

  Now take any two configurations $\sigma, \tilde{\sigma}$ 
  with $\dist(\sigma, \tilde{\sigma}) = k$.
  There is a sequence of states $\sigma_0, \ldots, \sigma_k$ such that
  $\sigma_0 = \sigma, \sigma_k = \tsigma$, and each neighboring
  pair $\sigma_{i}, \sigma_{i-1}$ are at unit distance.
  Since we proved the contraction holds for 
  configurations at unit distance, 
  \begin{equation*}
    \E_{\vec{\sigma}}\left[
      \dist(X_1^{\sigma},X_1^{\tsigma}) \right]
      \leq \sum_{i=1}^k \E_{\vec{\sigma}}\left[
	\dist(X_1^{\sigma_{i}},X_1^{\sigma_{i-1}}) \right]
      \leq \rho k = \rho \dist(\sigma, \tilde{\sigma}) .
  \end{equation*}
  This establishes \eqref{eq.hamcont} for $t=1$;
  iterating completes the proof.
\end{proof}

We mention another property of the monotone coupling,
from which it receives its name.  We write 
$\sigma \leq \sigma'$ to mean that $\sigma(i) \leq \sigma'(i)$
for all $i$.  Given the monotone coupling $(X_t, \tX_t)$,
if $X_t \leq \tX_t$, then $X_s \leq \tX_s$ for all
$s \geq t$.  This is obvious from the definition of
the grand coupling, since the function $p_+$ is 
non-decreasing.

\subsection{Magnetization chain}
Let $S_t \deq S(X_t)$, and note that $(S_t)$ is
itself a Markov chain on
$\Omega_{S} \deq \{ -1, -1 + 2/n, \ldots, 1 - 2/n, 1 \}$.
The increments $S_{t+1} - S_t$ take values in $\{-2/n,0,2/n\}$,
and the transition probabilities are
\begin{equation} \label{eq.Stm}
  P_M(s, s') =
  \begin{cases}
    \frac{1 + s}{2} p_-( s - n^{-1})  & s' = s - 2/n, \\
    \frac{1 - s}{2} p_+( s + n^{-1})  & s' = s + 2/n, \\
    1 - \frac{1 + s}{2} p_-( s - n^{-1})
      - \frac{1 - s}{2} p_+( s + n^{-1}) 
      & s' = s,
  \end{cases}
\end{equation}
for $s \in \Omega_{S}$, where $p_+(s)$ and $p_-(s)$ are as in
\eqref{eq.glauber-cw}. 

\begin{remark} \label{rmk.symmetric}
  It is easily verified that $P_M(-s, -s') = P_M(s, s')$, 
  so the distribution of the chain $(S_t)$ started from $s$ is the
  same as the distribution of $(-S_t)$ started from $-s$.
\end{remark}

\begin{remark} \label{rmk.truncated}
  Let $(X_t^+)$ be the Glauber dynamics restricted to
  $\Omega^+$, and define $S^+_t \deq S(X_t^+)$.
  The chain $(S^+_t)$ has the same 
  transition probabilities as the chain $|S_t|$.   
\end{remark}

In the remainder of this subsection, we collect some
facts about the Markov chain $(S_t)$ which will be
needed in our proofs.

If $(X_t, \tilde{X}_t)$ is a coupling of the Glauber dynamics,
we will always write $S_t$ and $\tilde{S}_t$ for
$S(X_t)$ and $S(\tX_t)$, respectively.

\begin{lemma} \label{lem.Scontr}
  Let $\rho$ be as defined in \eqref{eq.rhodefn}.
  If $(X_t, \tilde{X}_t)$ is the monotone coupling, started from
  states $\sigma$ and $\tilde{\sigma}$, then
  \begin{equation} \label{eq.Scontr}
    \E_{\sigma, \tsigma}\left[ |S_t - \tS_t| \right]
      \leq \left(\frac{2}{n}\right) \rho^t \dist(\sigma,\tsigma)
      \leq 2 \rho^t .
  \end{equation}
\end{lemma}
\begin{proof}
  Using the triangle inequality, we see that 
  $|S_t - \tS_t| \leq (2/n) \dist(X_t, \tX_t)$.
  An application of Proposition \ref{prop.hamcontr} completes the
  proof. 
\end{proof}

\begin{lemma} \label{lem.absS}
  For the magnetization chain $(S_t)$, for any two states
  $s$ and $\tilde{s}$ in $\Omega_S$ with $s \geq \tilde{s}$,
  \begin{align}
    0 \leq \E_s[S_1] - \E_{\tilde{s}}[S_1] & \leq \rho(s - \tilde{s}).
    \label{eq.EScoup}
  \intertext{Also, for any two states $s$ and $\tilde{s}$,}
  \left| \E_s[S_1] - \E_{\tilde{s}}[S_1] \right|
    & \leq \rho|s - \tilde{s}| . \label{eq.AEScoup}
  \end{align}
\end{lemma}
\begin{proof}
  Let $(X_t, \tilde{X}_t)$ be the monotone coupling,
  started from $(\sigma, \tsigma)$, where
  $\sigma \geq \tsigma$ and $S(\sigma) = s,
  S(\tsigma) = \tilde{s}$.   In this case,
  $s - \tilde{s} = (2/n)\dist(\sigma, \tsigma)$, and
  \begin{multline*}
    \E_{\sigma, \tsigma}[|S_1 - \tS_1|]
    = \E_{\sigma, \tsigma}[(2/n)\dist(X_1, \tX_1)] 
    \leq \frac{2}{n} \rho \dist(\sigma, \tsigma)
    = \rho (s - \tilde{s}) . 
  \end{multline*}
  By monotonicity, $X_1 \geq \tilde{X}_1$ and
  so $S_1 \geq \tilde{S}_1$.  Thus, 
  $\E_\sigma[S_1] - \E_{\tsigma}[\tS_1]
  = \E_{\sigma, \tsigma}[|S_1 - \tS_1|]$,
  which, together with the preceding inequality,
  proves that
  \begin{equation} \label{eq.EScoup-0}
    \E_\sigma[S_1] - \E_{\tsigma}[\tS_1]  
    \leq \rho(s - \tilde{s}).
  \end{equation}
  The left-hand side of
  \eqref{eq.EScoup-0} equals $\E_s[S_1] - \E_{\tilde{s}}[\tS_1]$,
  because $(S_t)$ is a Markov chain.
  Moreover, the left-hand side does not
  depends at all on the coupling.  This proves
  \eqref{eq.EScoup}.   An analogous bound in the case
  $S(\tsigma) \geq S(\sigma)$ establishes \eqref{eq.AEScoup}.
\end{proof}

\medskip
We now study the drift of $(S_t)$ in some detail.
From \eqref{eq.Stm},
\begin{equation*}
  \E[S_{t+1} - S_t \mid S_t = s]
    = \frac{2}{n} \left(\frac{1-s}{2} \right)p_+(s + n^{-1}) 
       - \frac{2}{n} \left( \frac{1+s}{2} \right) p_-(s - n^{-1}),
\end{equation*}
and hence
\begin{equation} \label{eq.deltaS1}
  \E[S_{t+1} - S_t \mid S_t = s] 
    = \frac{1}{n}\left[ f_n(s) - s + \theta_n(s) \right],
\end{equation}
where
\begin{align*}
  f_n(s)      & \deq \frac{1}{2}\left\{ \tanh[\beta(s+n^{-1})] 
                       + \tanh[\beta(s-n^{-1})] \right\} \\
  \theta_n(s) & \deq \frac{-s}{2}\left\{ \tanh[\beta(s+n^{-1})]
                       - \tanh[\beta(s-n^{-1})] \right\}.
\end{align*}
The approximation
\begin{equation} \label{eq.approxdeltaS1}
  \E[S_{t+1} - S_t \mid S_t=s] 
    \approx \frac{1}{n}\left[ \tanh(\beta s) - s \right]
\end{equation}
will play an important role in our proofs, and we will need
to control the error fairly precisely.   For the moment, let us
observe that  \eqref{eq.approxdeltaS1} is valid exactly as an
inequality for $s \geq 0$: 
\begin{equation} \label{eq.ineq}
  \E[S_{t+1} - S_t \mid S_t=s] 
    \leq \frac{1}{n}\left[ \tanh(\beta s) - s \right].
\end{equation}
This follows from the concavity of the hyperbolic tangent, together
with the fact that the term $\theta_n(s)$ in \eqref{eq.deltaS1} is
negative.
By Remark \ref{rmk.symmetric}, for $s \le 0$, 
\begin{equation} \label{eq.ineq-1}
  \E[S_{t+1} - S_t \mid S_t=s] 
    \geq \frac{1}{n}\left[ \tanh(\beta s) - s \right].
\end{equation}
Since $(S_t)$ does not change sign when $|S_t| > n^{-1}$,
and because $\tanh$ is an odd function,
putting together \eqref{eq.ineq} and \eqref{eq.ineq-1}
shows that, for $|S_t| > n^{-1}$,
\begin{equation} \label{eq.ineq-2}
  \E\left[ |S_{t+1}| \given S_t \right]
    \leq |S_t| 
       + \frac{1}{n}\left[ \tanh(\beta |S_t|) - |S_t| \right] . 
\end{equation}
Since $\tanh(x) \leq x$ for $x \geq 0$, when $\beta \leq 1$
equation \eqref{eq.ineq} implies that, for $s \geq 0$,
\begin{equation} \label{eq.driftposS}
  \E[S_{t+1} - S_t \mid S_t=s] 
    \leq \frac{s(\beta - 1)}{n} . 
\end{equation}

Define
\begin{equation} \label{eq.tau0defn}
  \tau_0 \deq \inf\{ t \geq 0 \st |S_t| \leq 1/n \} .
\end{equation}
Note that, for $n$ even, $|S_{\tau_0}| = 0$, while
for $n$ odd, $|S_{\tau_0}| = 1/n$. The notation $\tau_0$ will be
used with the same meaning throughout the paper. 

On several occasions, we will need an upper bound on the probability
that an unbiased random walk remains positive for at least $u$
steps. The following lemma gives a classical estimate. 

\begin{lemma} \label{lem.rwest}
  Let $(W_t)_{t \geq 0}$ be a random walk with
  $\E[W_{t+1}-W_t \mid W_t] = 0$ and $|W_{t+1}-W_t| < B$
  for some constant $B$.  Then there is a constant
  $c > 0$ such that, for all $u$,
  \begin{equation} \label{eq.rwest}
    \P_k( |W_1| > 0, \ldots, |W_u| > 0 ) \leq \frac{c |k|}{\sqrt{u}} .
  \end{equation}
  (Here $\P_k$ indicates probabilities for the random walk
  started with $W_0 = k$.)
\end{lemma}

Lemma \ref{lem.rwest} can be proved using hitting estimates in
\ocite{F:v2}; alternatively, it can be seen to be a special case 
of equation (3.9) in Bender, Lawler, Pemantle, and Wilf \ycite{BLPW}. 

The following lemma is proved for $n$ even.  The proof can be
modified to deal with the case of $n$ odd by replacing $0$ with $1/n$;
we omit the details.
\begin{lemma} \label{lem.martrp}
  Let $\beta \le 1$, and suppose that $n$ is even. 
  There exists a constant $c$ such that, for all $s$ and for all 
  $u,t \ge 0$,
  \begin{equation} \label{eq.martrp}
    \P( \, |S_{u}| > 0, \ldots, |S_{u+t}| > 0 \mid S_u=s )
      \leq \frac{c n |s|}{\sqrt{t}}.
  \end{equation}
\end{lemma}
\begin{proof}
  It will suffice to prove \eqref{eq.martrp} for $s > 0$,
  in which case the absolute values may be removed. 

  By \eqref{eq.driftposS}, $\E[S_{t+1} - S_t \mid S_t] \leq 0$
  for $S_t \geq 0$.   Also, there exists a constant $b > 0$ such that
  $\P( S_{t+1} - S_t \neq 0 \mid S_t ) \ge b$ for all times $t$,
  uniformly in $n$. It follows that $(S_t)$ can be coupled with 
  an unbiased nearest-neighbor random walk $(W_t)$ on $\Z$
  satisfying 
  \begin{itemize}
    \item $\P(W_1-W_0 \neq 0 \mid W_0 = w) = b$ for all $w$,
    \item $W_0 = ns/2$,
    \item $nS_t/2 \leq W_t$ for $t$ less than the
      first time $u$ when $S_u \leq n^{-1}$.  
  \end{itemize}
  From Lemma \ref{lem.rwest}, there exists a constant $c > 0$ such
  that  
  \begin{align*} 
    \P( S_{u+1} > 0, \ldots, S_{u+t} > 0 \mid S_u=s)
      & \leq \P ( W_1 > 0, \ldots, W_t > 0 \mid W_0 = ns/2) \\
      & \leq \frac{c n s}{\sqrt{t}}.
  \end{align*}
\end{proof}

\subsection{Variance bound}

\begin{lemma} \label{lem.variance2}
  Let $(Z_t)$ be a Markov chain taking values in $\R$ and
  with transition matrix $P$.  We will write  $\P_z$ and
  $\E_z$ for its probability measure and expectation, respectively,
  when $Z_0 = z$.   Suppose that there is some $0 < \rho < 1$ such
  that for all pairs of starting states $(z, \tilde{z})$, 
  \begin{equation} \label{eq.abscontr}
    \left| \, \E_z[Z_t] - \E_{\tilde{z}}[Z_t] \, \right|
    \leq \rho^t |z - \tilde{z}| .
  \end{equation}
  Then $v_t \deq \sup_{z_0} \var_{z_0} (Z_t)$ satisfies
  \begin{equation*}
    v_t \leq  v_1 \min\{t, (1-\rho^2)^{-1}\}.
  \end{equation*}
\end{lemma}
\begin{remark} \label{rmk.pc}
  Suppose that, for every pair $(z, \tilde{z})$, there is a
  coupling $(Z_1, \tilde{Z}_1)$ of $P(z, \cdot)$ and
  $P(\tilde{z}, \cdot)$ such that
  \begin{equation} \label{eq.contrarb}
    \E_{z, \tilde{z}}\left[\, |Z_1 - \tilde{Z}_1| \, \right]
      \leq \rho|z - \tilde{z}|.
  \end{equation}
  By iterating \eqref{eq.contrarb},
  \begin{equation*}
    \left| \E_z[ Z_t] - \E_{\tilde{z}}[ \tilde{Z}_t] \right|
    \leq \E_{z, \tilde{z}}[ |Z_t - \tilde{Z}_t| ]
    \leq \rho^t | z - \tilde{z} | .
  \end{equation*}
  The left-hand side does not depend at all on the coupling,
  and in particular, \eqref{eq.abscontr} holds.
  Moreover, if the state-space of $(Z_t)$ is discrete with a path
  metric and \eqref{eq.contrarb} holds for all neighboring pairs 
  $z,\tilde{z}$,  then it holds for all pairs of states;
  see~\ocite{BD:PC}. 
\end{remark}

\begin{proof}
  Let $(Z_t)$ and $(Z^\star_t)$ be
  \emph{independent} copies of the chain, both started from 
  $z_0$.  By the Markov property and \eqref{eq.abscontr}, 
  \begin{align*}
    \left| \, \E_{z_0}[Z_t \mid Z_1=z_1] 
       - \E_{z_0}[Z_t^\star \mid Z_1^\star = z_1^\star] \, \right|
      & = \left| \, \E_{z_1}[Z_{t-1}] 
            - \E_{z_1^\star}[Z^\star_{t-1}] \, \right| \\
      & \le \rho^{t-1} |z_1-z_1^\star|.
  \end{align*}
  Hence, letting $\phi (z) = \E_z [Z_{t-1}]$, we see that
  \begin{align}
    \var_{z_0}\left( \E_{z_0}[Z_t \mid Z_1] \right) 
       & = \frac{1}{2}\E_{z_0}\left[
             [ \phi (Z_1)- \phi (Z^\star_{1})]^2 
	     \right] \nonumber \\
       & \leq \frac{1}{2} \E_{z_0}\left[ \, \rho^{2(t-1)}
            |Z_1 - Z^\star_1|^2 \, \right] \nonumber \\
       & \leq v_1 \rho^{2(t-1)} \label{eq.condvar} .
  \end{align}
  By the ``total variance'' formula, for every $z_0$,
  \begin{equation*}
    \var_{z_0}(Z_t) 
       = \E_{z_0}\left[ \var_{z_0}(Z_t \mid Z_1) \right]
            + \var_{z_0}\left( \E_{z_0}[Z_t \mid Z_1] \right), 
  \end{equation*}
  so that
  \begin{equation} \label{eq.totvar}
     v_t 
       \le \sup_{z_0} \left\{ \E_{z_0}[ \var_{z_0}(Z_t \mid Z_1) ] 
         + \var_{z_0}\left( \E_{z_0}[Z_t \mid Z_1] \right) \right\}.
  \end{equation}
  Now, $\var_{z_0}( Z_t \mid Z_1 = z_1) \leq v_{t-1}$ for every $z_1$, 
  and so 
  \begin{equation} \label{eq.expvar}
    \E_{z_0}\left[ \var_{z_0}(Z_t \mid Z_1) \right] \leq v_{t-1}. 
  \end{equation}
  Thus we have shown that 
  $v_t \leq v_{t-1} + v_1 \rho^{2(t-1)}$,
  whence
  \begin{equation*}
    v_t \leq v_1 \sum_{i=0}^{t-1} \rho^{2(i-1)} 
      \leq v_1 \min\left\{ \; (1-\rho^2)^{-1}, \; t \right\}.
  \end{equation*}
\end{proof}

\begin{proposition} \label{prop.magvar}
  If $\beta < 1$, then $\var(S_t) = O(n^{-1})$ as $n \to \infty$.
  If $\beta = 1$, then $\var(S_t) = O(t/n^2)$ as $n \to \infty$.
\end{proposition}
\begin{proof}
  The conclusion follows from combining Lemma
  \ref{lem.absS} with Lemma \ref{lem.variance2},
  and observing that $v_1$ is bounded by $(4/n)^2$
  since the incrememnts of $(S_t)$ are at most $2/n$
  in absolute value.
\end{proof}

\subsection{Expected spin value}

In order to establish the cutoff at high temperature, not only do we
need to consider the magnetization chain, but also the number of
positive and negative spins among subsets of the vertices. 

\begin{lemma} \label{lem.spin-val}
  Let $\beta < 1$. 
  \begin{enumeratei}
    \item For all $\sigma \in \X$ and every $i=1,2,\ldots,n$,
      \begin{equation*}
	|\E_\sigma [S_t]| \le 2e^{-(1-\beta)t/n},
	\quad \text{and} \quad
	|\E_\sigma [X_t (i)]| \le 2e^{-(1-\beta)t/n} .
      \end{equation*}
    \item
      For any subset $A$ of vertices, if
      \begin{equation} \label{eq.MAdefn}
	M_t(A) \deq \frac{1}{2}\sum_{i \in A} X_t(i),
      \end{equation}
      then $|\E_\sigma[M_t(A)]| \leq |A|e^{-(1-\beta)t/n}$ and
      $\var( M_t(A) ) \leq cn$ for some constant $c > 0$.
    \item
      For any subset $A$ of vertices and all $\sigma \in \X$,
      \begin{equation} \label{eq.M}
	\E_\sigma\left[ |M_t(A)| \right] 
	  \leq n e^{-(1-\beta)t/n} + O(\sqrt{n}) .
      \end{equation}
  \end{enumeratei}
\end{lemma} 
\begin{proof}
  Let $\one$ denote the configuration of all plus spins,
  and let $(X_t^T, \tilde{X}_t)$ be the monotone coupling 
  with $X_0^T = \one$ and such that $\tilde{X}_0$ has
  distribution $\mu$.  (Note that then $\tilde{X}_t$ has
  distribution $\mu$ for all $t \geq 0$, by stationarity.)
  From Lemma \ref{lem.Scontr}, because
  $\E_\mu[\tilde{S}_t] = 0$, we have
  \begin{equation*} 
    \E_{\one}\left[S_t^T\right] 
      \leq \E_{\one, \mu}\left[ |S_t^T - \tilde{S}_t| \right]
             + \E_\mu\left[\tilde{S}_t\right] 
      \leq 2e^{-t(1-\beta)/n} .
  \end{equation*}
  By symmetry, 
  $\E_{\one}\left[X^T_t(i)\right] \le 2e^{-(1-\beta)t/n}$ for all  
  $i$.    By monotonicity, for any $\sigma$,
  \begin{equation*}
    \E_\sigma[X_t(i)] 
      \leq \E_{\one}[ X_t^T(i) ] 
      \leq 2e^{-(1-\beta)t/n} .
  \end{equation*}
  Because the chain $(-S_t)$ started from $-\sigma$ 
  has the same distribution as the chain $(S_t)$
  started from $\sigma$,
  \begin{equation*}
    -2e^{(1-\beta)t/n} \leq \E_{\sigma}[X_t(i)] .
  \end{equation*}

  For part~(ii), the bound on the expectation follows from~(i). As for
  the variance, since the spins are positively correlated, 
  \begin{equation}
    \var\left(\sum_{i\in A} X_t (i)\right)
      \le \var\left(\sum_{i=1}^n X_t (i)\right)
      \le n^2 \var(S_t) 
      \le c n,
  \end{equation}
  where the last inequality follows from
  Proposition~\ref{prop.magvar}. 

  For part (iii), let $(X_t, \tilde{X}_t)$ be the
  monotone coupling with $X_0 = \sigma$ and the
  distribution of $\tilde{X}_0$ equal to $\mu$.  
  From the triangle inequality,
  \begin{equation*}
    \E_{\sigma}[|M_t(A)|] 
      \leq \E_{\sigma,\mu}\left[|\tilde{M}_t(A) - M_t(A)|\right] 
        + \E_{\mu}\left[ |\tilde{M}_t(A)| \right].
  \end{equation*} 
  By the Cauchy-Schwartz inequality 
  and since
  $|\tilde{M}_t(A) - M_t(A)| \leq \dist(X_t, \tilde{X}_t)$,
  \begin{equation*}
    \E_{\sigma}[|M_t(A)|] 
      \leq \E_{\sigma,\mu}\left[ \dist(X_t, \tilde{X}_t) \right]
        + \sqrt{\E_{\mu}\left[ \tilde{M}_t(A)^2 \right]} .
  \end{equation*}
  Applying Proposition \ref{prop.hamcontr} shows that
  \begin{equation}
    \E_{\sigma}[|M_t(A)|] 
      \leq n \rho^t 
        + \sqrt{\E_{\mu}\left[ \tilde{M}_t(A)^2 \right]}. 
        \label{eq.absMA1}
  \end{equation}
  Since the variables $\{\tilde{X}_t(i)\}_{i=1}^n$ are positively
  correlated under $\mu$,
  \begin{equation} \label{eq.MAsq}
    \E_\mu\left[ \tilde{M}_t(A)^2 \right] 
      \leq \frac{n^2}{4} \E_{\mu}\left[ \tilde{S}_t^2 \right]
      =    \frac{n^2}{4} \var_\mu( \tilde{S}_t ) 
      =    O(n) ,
  \end{equation}
  where the last equality follows from Proposition \ref{prop.magvar}.
  Using \eqref{eq.MAsq} in \eqref{eq.absMA1} shows that
  \begin{equation}
    \E_{\sigma}\left[ |M_t(A)| \right] 
      \leq ne^{-(1-\beta)t/n} + O(\sqrt{n}) .
  \end{equation}
\end{proof}

\subsection{Coupling of chains with the same magnetization}

The following lemma holds at all temperatures, though we will only be
using it for $\beta \ge 1$.  It shows that once the magnetizations of
two copies of the Glauber dynamics agree, the two copies can be
coupled in such a way that the entire configurations agree after at
most another $O(n\log n)$ steps.   Note that this simple coupling is
not fast enough to show cutoff (where we need that once
the magnetizations agree, only order $n$ steps are required
to fully couple).  A more sophisticated coupling for this purpose
is given in Section~\ref{sec.cutoff}. 

\medskip
For any coupling $(X_t, \tX_t)$, we will let $\tau$
denote the coupling time:
\begin{equation*}
  \tau \deq \min\{ t \geq 0 \st X_t = \tilde{X}_t \}.
\end{equation*}

\begin{lemma} \label{lem.fromequalM}
  Let $\sigma,\tilde{\sigma} \in \Omega$ be such that 
  $S(\sigma ) = S(\tilde{\sigma})$.   There exists
  a coupling $(X_t, \tilde{X}_t)$ of the Glauber dynamics
  with initial states $X_0 = \sigma$ and $\tilde{X}_0 =
  \tilde{\sigma}$ such that
  \begin{equation*}
    \limsup_{n \rightarrow \infty}
     \P_{\sigma, \tilde{\sigma}}( \tau > c_0(\beta) n \log n )
       = 0, 
  \end{equation*}
  for some constant $c_0(\beta)$ large enough.
\end{lemma}
\begin{proof}
  To update the configuration $X_t$ at time $t$, proceed as
  follows:  Pick a site $I \in \{1,2,\ldots,n\}$ uniformly at random,
  and generate a random spin $\spin$ according to 
  \begin{equation*}
    \spin =
    \begin{cases}
      +1 & \text{with probability } p_+(S_t - X_t(I)/n),\\
      -1 & \text{with probability } p_-(S_t - X_t(I)/n).
    \end{cases}
  \end{equation*}
  Set
  \begin{equation*}
    X_{t+1}(i) 
    =
    \begin{cases}
      X_t(i) & i \neq I, \\
      \spin & i = I .
    \end{cases}
  \end{equation*}
  As for updating $\tilde{X}_t$, 
  if $X_t(I) = \tilde{X}_t(I)$, then let
  \begin{equation*}
    \tilde{X}_{t+1}(i) =
    \begin{cases}
      \tilde{X}_t(i) & i \neq I, \\
      \spin & i = I.
    \end{cases}
  \end{equation*}
  If $X_t(I) \neq \tilde{X}_t(I)$, then we pick a vertex $\tilde{I}$
  uniformly at random from the set
  \begin{equation*}
    \{ i \st \tilde{X}_t(i) \neq X_t(i), \text{ and } 
       \tilde{X}_t(i) = {X}_t(I) \},
  \end{equation*}
  and set
  \begin{equation*}
    \tilde{X}_{t+1}(i) =
    \begin{cases}
      \tilde{X}_t(i) & i \neq \tilde{I}, \\
      \spin & i = \tilde{I} .
    \end{cases}
  \end{equation*}
  Let $D_t = \sum_{i=1}^n|X_t(i) - \tilde{X}_t(i)|/2$ be the
  number of differing coordinates between $\tilde{X}_t$ and
  $X_t$.
  
  There exists a constant $c_1 = c_1(\beta) > 0$ such that 
  $p_+(s) \wedge p_-(s) \geq c_1$ uniformly over all 
  $s \in \{-1, \ldots, 1\}$ and all $n$.
  If $X_{t}(I) = \tilde{X}_{t}(I)$, then $D_{t+1} - D_t = 0$ 
  while if $X_{t}(I) \neq \tilde{X}_{t}(I)$, 
  then $D_{t+1} - D_t = -2$.
  It follows that
  \begin{equation*}
    \E[ D_{t+1} - D_t \mid X_t, \tilde{X}_t ] \leq - \frac{2c_1 D_t}{n},
  \end{equation*}
  so $Y_t = D_t(1-2c_1/n)^{-t}$ is a non-negative supermartingale,
  whence
  \begin{equation*}
    \E[D_t ] 
    \leq \E[D_0]\left( 1 - \frac{2c_1}{n} \right)^t
    \leq n e^{-2c_1t/n} . 
  \end{equation*}

  Taking $t = c_0 n \log n$ for a sufficiently large constant 
  $c_0 = c_0(\beta)$,
  we can make the right hand side less than $1/n$, say. 
  Markov's inequality yields
  \begin{equation*}
    \P_\sigma\left( \tau > c_0 n \log n \right)
      \leq \P_\sigma\left( D_{c_0 n \log n} \geq 1 \right) 
      \leq \E_{\sigma}[D_{c_0n\log n}] \leq \frac{1}{n} .
  \end{equation*}
\end{proof}

\section{Cutoff for the Glauber dynamics at high temperature}
\label{sec.cutoff}

In this section we prove Theorem~\ref{thm.hightemp}. 
As always, $(X_t)$ will denote the Glauber dynamics, and 
$S_t = S(X_t) = n^{-1} \sum_{i=1}^n X_t (i)$ is the normalized
magnetization chain.   Recall the definitions
\begin{align*}
  t_n & = [2(1-\beta)]^{-1}n\log n, \\
  \rho & = 1-(1-\beta)/n, \\
  \tau_0 & = \min\{ t \geq 0 \st |S_t| \leq 1/n \} .
\end{align*}

\subsection{Upper bound} \label{sec.cutoff-upper}

For convenience, we restate the 
upper bound part of Theorem~\ref{thm.hightemp}:
\begin{theorem} \label{thm.hitempub}
  If $\beta < 1$, then
  \begin{equation} \label{eq.hitempub}
    \lim_{\gamma \rightarrow \infty} \limsup_{n \rightarrow \infty}
    d_n\left(\left[{2(1-\beta)}\right]^{-1} n \log n + \cn n\right)
    = 0 .
  \end{equation}
\end{theorem}

Our strategy is to first construct a coupling of the dynamics so
that the magnetizations agree with high probability
after $t_n + O(n)$ steps.  

\begin{lemma} \label{lem.magcouple}
  Let $\sigma$ and $\tilde{\sigma}$ be any two configurations.  
  There is a coupling $(X_t, \tilde{X}_t)$ of the Glauber dynamics  
  with $X_0 = \sigma$ and $\tilde{X}_0 = \tilde{\sigma}$
  such that, if 
  \begin{equation} \label{eq.taumagdef}
    \taumag \deq \min\{ t \geq 0 \st S_t = \tS_t \},
  \end{equation}
  then for some constant $c > 0$ not depending on
  $\sigma,\tilde{\sigma}$ or $n$,
  \begin{equation} \label{eq.maghitbd}
    \P_{\sigma,\tsigma}( \taumag > t_n  +  \cn n ) 
      \leq \frac{c}{\sqrt{\gamma}} .  
  \end{equation}
\end{lemma}
\begin{proof}
  Assume without loss of generality that $S(\sigma) > S(\tsigma)$.
  Let $(X_t, \tilde{X}_t)$ be the monotone coupling of
  Section \ref{sec.mon-coupling}.  Define
  $\Delta_t \deq (n/2)|S_t - \tilde{S}_t|$.
  By Lemma \ref{lem.Scontr}, for some $c_1 > 0$,
  \begin{equation} \label{eq.sts}
      \E_{\sigma, \tsigma}\left[ \Delta_{t_n} \right] 
      \leq c_1 \sqrt{n} .
  \end{equation}
  Define $\tau_1 \deq \min\{ t \geq t_n \st |\Delta_t| \leq 1 \}$.
  For $t_n \leq t < \tau_1$, allow $(X_t)$ and 
  $(\tilde{X}_t)$ to run independently.  

  Since $S_t \geq \tS_t$ for $t \leq \tau_1$, from 
  Lemma \ref{lem.absS}, the process  
  $(S_t - \tS_t)_{t_n \leq t < \tau_1}$ has non-positive drift.  
  Moreover,  
  since $(X_t)_{t_n \leq t < \tau_1}$ and $(\tX_t)_{t_n \leq t <
  \tau_1}$ are independent given $X_{t_n}, \tX_{t_n}$, 
  for $t > t_n$ the conditional probability that $S_t - \tS_t$
  is non-zero is bounded away from zero uniformly.  Thus there is
  a random walk $(W_t)_{t \geq t_n}$ defined on the
  same probability space as $(X_{t}, \tX_t)$ and satisfying:
  the increments $W_{t+1} - W_t$ are mean-zero and bounded, 
  $n(S_t - \tS_t) \leq W_t$ on  $[t_n, \tau_1)$, and 
  $n(S_{t_n} - \tS_{t_n}) = W_{t_n}$.

  By Lemma \ref{lem.rwest},
  \begin{align*}
    \P_{\sigma, \tsigma}( \tau_1 > t_n + \gamma n \mid X_{t_n}, \tX_{t_n})
    & \leq \P_{\sigma, \tsigma}( W_{t_n+1} > 0, 
      \ldots, W_{t_n + \gamma n} > 0 \mid X_{t_n}, \tX_{t_n}) \\
    & \leq \frac{ n |S_{t_n} - \tS_{t_n}| }{ \sqrt{\gamma n} } .
  \end{align*}
  Taking expectation above, \eqref{eq.sts} shows that
  \begin{equation*}
    \P_{\sigma, \tsigma}( \tau_1 > t_n + \gamma n ) 
      \leq O\left(\gamma^{-1/2} \right) .
  \end{equation*}
  The number of plus spins in $X_{\tau_1}$ is either
  one more than, or the same as, the number of plus spins
  in $\tX_{\tau_1}$.  Match each plus spin in $\tX_{\tau_1}$ with a plus
  spin in $X_{\tau_1}$, and match the remaining spins
  arbitrarily.  From time $\tau_1$ onwards, run a modified version of
  the monotone coupling, where matched vertices are updated
  together in the two chains.  Define $\dist'$ as the
  number of disagreements between matched vertices.   The conclusion
  of Lemma \ref{lem.Scontr} now holds for this modified
  monotone coupling, with the distance $\dist'$ replacing
  $\dist$ in \eqref{eq.Scontr}.
  Thus,
  \begin{align*}
    \P_{\sigma, \tsigma}( \taumag > \tau_1 + \gamma'n \mid 
      X_{\tau_1}, \tX_{\tau_1} )
    & \leq \P_{\sigma, \tsigma}( \Delta_{\tau_1 + \gamma'n} > 1
        \mid    X_{\tau_1}, \tX_{\tau_1} ) \\
    & \leq \E_{\sigma, \tsigma}[ \Delta_{\tau_1 + \gamma'n}
	\mid   X_{\tau_1}, \tX_{\tau_1} ]\\
    & \leq \left(1 - \frac{1-\beta}{n} \right)^{\gamma' n}\\
    & \leq e^{-(1-\beta)\gamma'}.
  \end{align*}
  We conclude that
  \begin{equation*}
    \P_{\sigma, \tsigma}(\taumag \leq t_n + \gamma n + \gamma'n) \geq 
    1 - O\left(\gamma^{-1/2}\right) .
  \end{equation*}
\end{proof}

\subsection{Good starting states}

To show the cut-off upper bound, we will start by running the Glauber 
dynamics for an initial burn-in period.  This will ensure
that the chain is with high probability in a 
`nice' configuration required for the coupling argument in
Section \ref{sec.tcc}.
The following lemma is required: 

\begin{lemma} \label{lem.goodstart}
  For any a subset $\Omega_0 \subset \Omega$,
  \begin{equation} \label{eq.goodstart}
    \begin{split}
      d(t_0 + t) 
      & = \max_{\sigma \in \X} \| \P_{\sigma} (X_{t_0 + t} \in \cdot)  
          - \pi \|_{{\rm TV}} \\
      & \leq \max_{\sigma_0 \in \X_0}
	  \| \P_{\sigma_0} (X_t \in \cdot ) - \pi\|_{{\rm TV}} 
          + \max_{\sigma \in \Omega} \P_\sigma( 
	      X_{t_0} \not\in \Omega_0). 
    \end{split}
  \end{equation}
\end{lemma}
\begin{proof}
  For $A \subset \X$, we can bound 
  $|\P_{\sigma}(X_{t_0 + t} \in A) - \pi(A)|$ above by
  \begin{multline*}
    \Bigg|
    \sum_{\sigma_0 \in \Omega_0} 
    \left[
      \P_{\sigma}(X_{t_0 + t} \in A \mid X_{t_0} = \sigma_0)
      - \pi(A)
    \right] 
    \P_{\sigma}(X_{t_0} = \sigma_0) \\
    + \left[
      \P_{\sigma}(X_{t_0 + t} \in A \mid X_{t_0} \not\in \Omega_0)  
      - \pi(A)
      \right]
      \P_{\sigma}(X_{t_0} \not\in \Omega_0) 
      \, \Bigg| .
  \end{multline*}
  Using the triangle inequality, the preceding displayed quantity is
  bounded above by 
  \begin{equation*}
    \sum_{\sigma_0 \in \Omega_0} 
      | \P_{\sigma}(X_{t_0 + t} \in A \mid X_{t_0} = \sigma_0)
      - \pi(A) | \P_{\sigma}(X_{t_0} = \sigma_0) 
      + \P_\sigma(X_{t_0} \not\in \Omega_0) .
  \end{equation*}
  Taking a maximum over subsets $A$ shows that
  \begin{multline*}
    \| \P_{\sigma}(X_{t_0 + t} \in \cdot) - \pi \|_{{\rm TV}}
    \\
    \leq \sum_{\sigma_0 \in \Omega_0} 
    \| \P_{\sigma}(X_{t_0 + t} \in \cdot \mid X_{t_0} = \sigma_0) 
    - \pi \|_{{\rm TV}}
    \P_\sigma(X_{t_0} = \sigma_0) + \P_\sigma(X_{t_0} \not\in \Omega_0).
  \end{multline*}
  By the Markov property, $\P_\sigma(X_{t_0+t}\in \cdot \mid
  X_{t_0} = \sigma_0) = \P_{\sigma_0}(X_{t} \in \cdot)$,
  and bounding the average above by the maximum term 
  yields
  \begin{equation*}
    \| \P_{\sigma}(X_{t_0 + t} \in \cdot) - \pi \|_{{\rm TV}} 
    \leq 
    \max_{\sigma_0 \in \Omega_0} 
    \| \P_{\sigma_0}(X_{t} \in \cdot)
    - \pi \|_{{\rm TV}}
    + \P_\sigma(X_{t_0} \not\in \Omega_0) .
  \end{equation*}
  Taking a maximum over $\sigma \in \Omega$ 
  establishes \eqref{eq.goodstart}.
\end{proof}

In the proof of Theorem \ref{thm.hitempub}, we apply Lemma
\ref{lem.goodstart} with 
\begin{equation*}
  \X_0 = \{ \sigma \in \X \st |S(\sigma)| \leq 1/2 \} .
\end{equation*}

For a configuration $\sigma_0 \in \Omega$ define
\begin{equation*}
  \bar{u}_0 \deq |\{i \st \sigma_0(i)=1\}|, \quad
  \bar{v}_0 \deq |\{i \st \sigma_0(i)=-1\}|, 
\end{equation*}
the number of positive and negative spins, respectively,
in $\sigma_0$.  Also,
define
$\Lambda_0 \deq \{ (u,v) \st n/4 \leq u,v \leq 3n/4\}$.
Note that 
\begin{equation} \label{eq.omegalambda}
  \sigma_0 \in \Omega_0 \text{ if and only if }
  (\bar{u}_0, \bar{v}_0) \in \Lambda_0 .
\end{equation}

By Lemma \ref{lem.spin-val}, there is a constant $\theta_0 >0$
such that $|\E_\sigma[S_{\theta_0 n}]| \leq 1/4$,
whence, for $n$ large enough,  
\begin{align} \label{eq.Xgoodstart}
  \P_{\sigma}(X_{ \theta_0 n} \not\in \X_0) 
  & =  \P_{\sigma}( |S_{\theta_0 n}| > 1/2 ) \nonumber \\
  & \leq  \P_\sigma\left( 
               \left|S_{\theta_0 n}
	        - \E_\sigma[S_{\theta_0 n}] \right| 
		> 1/4 \right) \nonumber \\
  & \leq   16 \var_{\sigma} (S_{\theta_0 n}) = O(n^{-1}). 
\end{align}
The last equality follows from Proposition \ref{prop.magvar}.

\subsection{Two-coordinate chain} \label{sec.tcc}

Fix a configuration $\sigma_0 \in \Omega_0$.
For $\sigma \in \Omega$, define
\begin{align*}
U_{\sigma_0}(\sigma) 
            & \deq | \{ i \in \{0,1,\ldots, n\} \st \sigma(i) 
	    = \sigma_0(i) = 1 \} | \\
V_{\sigma_0}(\sigma) 
            & \deq | \{ i \in \{0,1,\ldots, n\} \st \sigma(i) 
	    = \sigma_0(i) = -1 \} | .
\end{align*} 
In what follows, we shall usually omit the subscript, writing
simply $U(\sigma)$ for $U_{\sigma_0}(\sigma)$ and 
$V(\sigma)$ for $V_{\sigma_0}(\sigma)$.  

For a copy of the Glauber dynamics $(X_t)$, the process
$(U_t, V_t)_{t \geq 0}$ defined by
\begin{equation} \label{eq.UVDefn}
  U_t = U(X_t), \quad \text{and} \quad V_t = V(X_t)
\end{equation}
is a Markov chain on $\{0,1,\ldots,u_0\} \times \{0,1,\ldots,v_0\}$
(with transition probabilities depending on the designated
configuration $\sigma_0$). We will refer to the chain $(U_t,V_t)$ as
the \emph{two-coordinate chain}, and its stationary measure
will be denoted by $\pi_2$. Note also that $(U_t,V_t)$ determines the
magnetization chain, as we can write 
\begin{equation} 
  S_t = \frac{2(U_t - V_t)}{n} - \frac{\bar{u}_0 - \bar{v}_0}{n} .
\end{equation}

It turns out that, by symmetry, the distance of the law of $X_t$ to 
$\mu$ equals the distance of the law of $(U_t,V_t)$ to $\pi_2$, as
established in the following lemma:

\begin{lemma} \label{lem.tvequal}
  If $(X_t)$ is the Glauber dynamics started from $\sigma_0$
  and $(U_t, V_t)$ is the chain defined by \eqref{eq.UVDefn}
  started from $(\bar{u}_0, \bar{v}_0)$, then
  \begin{equation} \label{eq.tvequal}
    \| \P_{\sigma_0}(X_t \in \cdot ) - \mu \|_{{\rm TV}}
    = \| \P_{(\bar{u}_0, \bar{v}_0)}(
         (U_t, V_t) \in \cdot )- \pi_2 \|_{{\rm TV}}.  
  \end{equation}
\end{lemma}
\begin{proof}
  Let
  \begin{equation*}
    \Omega(u,v) \deq 
      \{ \sigma \in \Omega \st (U(\sigma), V(\sigma)) = (u,v) \}.
  \end{equation*}
  Since both $\mu( \cdot \mid \X(u,v))$ and 
  \begin{equation*}
    \P_{\sigma_0} (X_t \in \cdot \mid (U_t, V_t) = (u,v) )
  \end{equation*}
  are uniform over $\Omega(u,v)$, it follows that
  \begin{multline*}
    \P_{\sigma_0}(X_t = \eta) - \mu(\eta) \\
      =  \sum_{u,v} \frac{\one\{ \eta \in \Omega(u,v)\}}{|\Omega(u,v)|}
      \left[ \P_{\sigma_0}( (U_t,V_t) = (u,v) ) 
	- \mu( \Omega(u,v) ) \right] .
  \end{multline*}
  Applying the triangle inequality, summing over $\eta$, and
  changing the order of summations shows that
  \begin{equation*} 
    \| \P_{\sigma_0}(X_t \in \cdot ) - \mu \|_{{\rm TV}}
      \leq \| \P_{(\bar{u}_0, \bar{v}_0)}(
          (U_t, V_t) \in \cdot )- \pi_2 \|_{{\rm TV}}.  
  \end{equation*}
  The reverse inequality holds since $(U_t, V_t)$ is
  a function of $(X_t)$.
\end{proof}

Identity~\eqref{eq.tvequal} implies that it suffices to bound
from above the distance to stationarity of the two-coordinate chain.

\begin{lemma} \label{lem.uvcoupling}
  Suppose two configuration $\sigma$ and $\tilde{\sigma}$ satisfy
  $S(\sigma) = S(\tilde{\sigma})$ and 
  $R_0 = U(\tilde{\sigma})-U(\sigma)>0$. 
  Define
  \begin{equation} \label{eq.g2def}
    \Xi_1 \deq \{ \sigma \st \min\{ U(\sigma),\, \bar{u}_0 - U(\sigma),\,
                V(\sigma),\, \bar{v}_0 - V(\sigma)\} \geq n/16 \}.
  \end{equation}
  There exists a Markovian coupling $(X_t, \tilde{X}_t)$ 
  of the Glauber dynamics with starting states
  $X_0 = \sigma$ and $\tilde{X}_0 = \tilde{\sigma}$ 
  such that the following hold:
  \begin{enumeratei}
    \item $S(X_t) = S(\tilde{X}_t)$ for all $t \geq 0$.
    \item If $R_t \deq U(\tilde{X}_{t}) - U(X_t)$ and
      $R_0 \geq 0$, then $R_t \ge 0$ and
      for all $t$ and
      \begin{equation} \label{eq.rdrift}
	\E_{\sigma, \tsigma}\left[ R_{t+1} - R_t \mid
	  X_t, \tX_t \right] 
	\leq 0 .
      \end{equation}
    \item There exists a constant $c$ not depending on $n$ so
      that on the event $\{X_t \in \Xi_1,\, \tX_t \in \Xi_1\}$, 
      \begin{equation} \label{eq.prmove}
	\P_{\sigma,\tsigma}(R_{t+1} - R_t \neq 0 \mid X_t, \tX_t ) \geq c.
      \end{equation}
  \end{enumeratei}
\end{lemma}
\begin{proof}
  Given the coupling $(X_t, \tilde{X}_t)$, 
  we define $\tilde{U}_t \deq U(\tX_t)$ and $\tilde{V}_t \deq
  V(\tX_t)$, and note that
  $\tilde{U}_t = U_t + R_t$ and $\tilde{V}_t = V_t + R_t$.
  
  For any configuration $\sigma$, we divide the vertices into four sets:
  \begin{equation} \label{eq.setdefns}
  \begin{split}
    A(\sigma) & = \{ i \in \{1, 2,\ldots n\} \st \sigma_0(i) = +1, \; 
                     \sigma(i) = +1 \},\\
    B(\sigma) & = \{ i \in \{1, 2,\ldots n\} \st \sigma_0(i) = +1, \; 
                     \sigma(i) = -1 \}, \\
    C(\sigma) & = \{ i \in \{1, 2,\ldots n\} \st \sigma_0(i) = -1, \; 
                     \sigma(i) = +1 \}, \\
    D(\sigma) & = \{ i \in \{1, 2,\ldots n\} \st \sigma_0(i) = -1, \; 
                     \sigma(i) = -1 \},
  \end{split}
  \end{equation}
  and so
  \begin{equation*}
    |A(\sigma)| = U(\sigma), \; |B(\sigma)| = \bar{u}_0 - U(\sigma), \;
    |C(\sigma)| = \bar{v}_0 - V(\sigma), \; |D(\sigma)| = V(\sigma) .
  \end{equation*}
  See Figure \ref{fig.subsets} for a schematic representation of this
  partition for $X_t$ and $\tilde{X}_t$.

  \begin{figure}
    \begin{center} 
      \begin{tiny}
      \begin{tabular*}{\columnwidth}{@{\extracolsep{\fill}}
	  |l|llllll|lllllll|} 
	\multicolumn{1}{c}{} & 1 & 2 & 3 & \multicolumn{9}{c}{$\cdots
	  \cdots \cdots$} & \multicolumn{1}{c}{$n$} \\
	\hline
	$\sigma_0$ & + & + & + & + & + & + & - & - & - & - & - & - & - \\
	& \multicolumn{6}{c|}{$u_0$} & \multicolumn{7}{c|}{$v_0$} \\
	\hline 
	$X_t$ & + & + & + \vline & - & - & - & + & + & + & + \vline 
          & - & - &  - \\
	& \multicolumn{3}{c}{\hfill $A(X_t)$ \hfill \vline}  
	& \multicolumn{3}{c|}{$B(X_t)$} 
	& \multicolumn{4}{c}{\hfill $C(X_t)$ \hfill \vline }  
	& \multicolumn{3}{c|}{$D(X_t)$} \\
	\hline 
	$\tilde{X}_t$ & + & + & + & + \vline & - & - & + & + & +
          \vline & - & - & - &  - \\
        & \multicolumn{4}{c}{\hfill $A(\tilde{X}_t)$ \hfill \vline}  
	& \multicolumn{2}{c|}{$B(\tilde{X}_t)$} 
	& \multicolumn{3}{c}{\hfill $C(\tilde{X}_t)$ \hfill \vline}  
        & \multicolumn{4}{c|}{$D(\tilde{X}_t)$} \\
	\hline 
      \end{tabular*}
      \end{tiny}
      \caption{The vertices in $X_t$ and $\tilde{X}_t$ are
	partitioned into four categories.} \label{fig.subsets}
    \end{center}
  \end{figure}
    
  Our coupling is as follows:
  To update $X_t$, select a uniformly random $I \in \{1,2 \ldots, n\}$, 
  and generate a random spin $\spin$ for $I$ according to the
  distribution
  \begin{equation*}
    \spin = 
    \begin{cases}
      +1 & \text{with probability } p_+(S_t - X_t(I)/n), \\
      -1 & \text{with probability } p_-(S_t - X_t(I)/n).
    \end{cases}
  \end{equation*}
  Set
  \begin{equation*}
    X_{t+1}(i) 
    =
    \begin{cases}
      X_t(i) & i \neq I, \\
      \spin & i = I.
    \end{cases}
  \end{equation*}
  For $\tilde{X}_t$, we select $\tilde{I}$ uniformly at random from
  $\{ i \st \tilde{X}_t(i) = X_t(I) \}$,
  and let
  \begin{equation*}
    \tilde{X}_{t+1}(i) =
    \begin{cases}
      \tilde{X}_t(i) & i \neq \tilde{I}, \\
      \spin & i = \tilde{I} .
    \end{cases}
  \end{equation*}
  The difference $R_{t+1} - R_t$ is determined by
  the values of $I, \tilde{I}$ and $\spin$ according
  to the following table:
  \begin{equation*}
    \begin{array}{llr|r}
      I & \tilde{I} & \spin & R_{t+1} - R_t \\
      \hline
      I \in B(X_t) & \tilde{I} \in D(\tilde{X}_t) & +1 & -1 \\
      I \in C(X_t) & \tilde{I} \in A(\tilde{X}_t) & -1 & -1 \\
      I \in A({X}_t) & \tilde{I} \in C(\tilde{X}_t) & -1 & + 1 \\
      I \in D(X_t) & \tilde{I} \in B(\tilde{X}_t) & +1 & + 1 \\
      \multicolumn{3}{l|}{\text{all other combinations}} & 0 \\
    \end{array}
  \end{equation*}
  It follows that
  \begin{align*}
    \P_{\sigma, \tsigma} ( R_{t+1} - R_t = -1
      \mid X_t, \tX_t ) & = a(U_t, V_t, R_t),\\
    \P_{\sigma, \tsigma} ( R_{t+1} - R_t = +1 
      \mid X_t, \tX_t ) & = b(U_t, V_t, R_t), 
  \end{align*}
  where (using the identities $\tilde{U}_t = U_t + R_t$ and
  $\tilde{V}_t = V_t + R_t$)
  \begin{align*}
    a(U_t, V_t, R_t) 
    & = 
      \left( \frac{\bar{v}_0 - V_t}{n} \right)
      \left( \frac{U_t + R_t}{\bar{v}_0 + U_t - V_t} \right)
      p_-(S_t - 1/n)  \\
    & \quad +
      \left( \frac{\bar{u}_0 - U_t}{n} \right)
      \left( \frac{V_t + R_t}{\bar{u}_0 - U_t + V_t} \right)
      p_+(S_t + 1/n),  \\
      b(U_t, V_t, R_t) 
    & =
      \left( \frac{U_t}{n} \right)
      \left( \frac{\bar{v}_0 - V_t - R_t}{
	\bar{v}_0 + U_t - V_t} \right) 
      p_-(S_t - 1/n) \\
    & \quad +
      \left( \frac{V_t}{n} \right)
      \left( \frac{\bar{u}_0 - U_t - R_t}{ 
	\bar{u}_0 - U_t + V_t} \right)
      p_+(S_t + 1/n) .
  \end{align*}
  We obtain
  \begin{align*}
    \E_{\sigma, \tsigma}\left[ R_{t+1} - R_t \mid X_t, \tX_t \right] 
    & = b(U_t, V_t, R_t)  - a(U_t, V_t, R_t)  \\
    & = \frac{-R_t}{n}\left[ p_-(S_t - 1/n) + p_+(S_t + 1/n) \right],
  \end{align*}
  so, in particular,
  \begin{equation} \label{eq.rnegdrift}
    \E_{\sigma, \tsigma}[ R_{t+1} - R_t \mid X_t, \tX_t ] \leq 0 .
  \end{equation}
  Furthermore, on the event $\{X_t \in \Xi_1, \tilde{X}_t \in \Xi_1 \}$,
  \begin{equation*}
    \P_{\sigma, \tsigma} (R_{t+1} - R_t \neq 0 \mid X_t, \tX_t ) \geq
    b(U_t, V_t, R_t)  \geq c 
  \end{equation*}
  for some constant $c>0$, uniformly in $n$, since the functions
  $p_+$ and $p_-$ are uniformly bounded away from $0$ and $1$. 
\end{proof}

\begin{proof}[Proof of Theorem \ref{thm.hitempub}] 
  Applying Lemma \ref{lem.goodstart} with $t_0 = \theta_0n$,
  together with the bound \eqref{eq.Xgoodstart}, shows that
  \begin{align} 
    d_n(\theta_0n + t) 
      &  \leq \max_{\sigma_0 \in \X_0}
	    \| \P_{\sigma_0}( X_{t} \in \cdot ) 
	       - \mu \|_{{\rm TV}} + O(n^{-1}).
	       \label{eq.agoodstart}
  \end{align} 
  Hence, using Lemma \ref{lem.tvequal} and \eqref{eq.omegalambda},
  \begin{equation}  \label{eq.firstbit}
    d_n(\theta_0n + t) 
      \le  \max_{(\bar{u}_0,\bar{v}_0) \in \Lambda_0} 
            \| \P_{(\bar{u}_0, \bar{v}_0)}( (U_t,V_t) \in \cdot) 
            - \pi_2 \|_{{\rm TV}} + O(n^{-1}), 
  \end{equation}
  recalling that 
  $\Lambda_0 = \{ (u,v) \st n/4 \leq u, v \leq 3n/4 \}$.

  We will call a pair of chains $(U_t, V_t)_{t \geq 0}$ and
  $(\tilde{U}_t, \tilde{V}_t)_{t \geq 0}$ a \emph{coupling
  of the two-coordinate chain} with initial values
  $(\bar{u}_0, \bar{v}_0)$ and $(\tilde{u}, \tilde{v})$ if
  \begin{itemize}
    \item The two chains are defined on a common probability space,
    \item Each of $(U_t, V_t)$ and $(\tilde{U}_t, \tilde{V}_t)$
      has the same transition probabilities as $(U(X_t), V(X_t))$, where
      $(X_t)$ is the Glauber dynamics,
    \item $(U_0, V_0) = (\bar{u}_0, \bar{v}_0)$ 
      and $(\tilde{U}, \tilde{V}) = (\tilde{u}, \tilde{v})$.
  \end{itemize}
  We will always consider couplings which 
  have $(\bar{u}_0, \bar{v}_0) \in \Lambda_0$, but
  $(\tilde{u}, \tilde{v})$ will not be so constrained.
  
  For a given coupling of the two-coordinate chain as above, we let
  \begin{equation*}
    \tau_c 
      \deq \min\{ t \geq 0 \st (U_t, V_t) 
             = (\tilde{U}_t, \tilde{V}_t) \} . 
  \end{equation*}

  For a coupling with initial states
  $(\bar{u}_0, \bar{v}_0)$ and $(\tilde{u}, \tilde{v})$,
  \begin{equation} \label{eq.coupineq}
    \| \P_{\bar{u}_0,\bar{v}_0}\left( (U_t, V_t) \in \cdot \right) 
        - \P_{\tilde{u},\tilde{v}}\left( 
	  (\tilde{U}, \tilde{V}_t) \in \cdot \right)
    \|_{{\rm TV}} 
      \leq \P_{(\bar{u}_0, \bar{v}_0), (\tilde{u}, \tilde{v})}
            ( \tau_c > t ) .
  \end{equation}
  (See, for example, \ocite{L:CM}*{Equation 2.8}.)
  A simple calculation shows that
  \begin{multline} \label{eq.coupave}
    \max_{(\bar{u}_0,\bar{v}_0) \in \Lambda_0}
    \| \P_{\bar{u}_0,\bar{v}_0}( (U_t, V_t) \in \cdot) 
      - \pi_2 \|_{{\rm TV}} \\
    \leq \max_{\substack{(\bar{u}_0, \bar{v}_0) \in \Lambda_0,\\ 
      (\tilde{u},\tilde{v})}} 
      \| \P_{\bar{u}_0,\bar{v}_0}( (U_t, V_t) \in \cdot )
      - \P_{\tilde{u}, \tilde{v}}( (\tilde{U}_t, \tilde{V}_t ) 
        \in \cdot)\|_{{\rm TV}} .
  \end{multline}
  We say that $f(n,t)$ is a \emph{uniform coupling bound} if
  for any initial states $(\bar{u}_0, \bar{v}_0) \in \Lambda_0$ and
  $(\tilde{u}, \tilde{v})$, there is a 
  coupling of the two-coordinate chain with
  \begin{equation*}
    \P_{(\bar{u}_0, \bar{v}_0), (\tilde{u}, \tilde{v})}
      ( \tau_c > t ) \leq f(n,t) .
  \end{equation*}
  If $f(n,t)$ is a uniform coupling bound, then
  combining \eqref{eq.coupineq} with
  \eqref{eq.coupave} shows that
  \begin{equation*}
    \max_{(\bar{u}_0, \bar{v}_0) \in \Lambda_0}
    \|\P_{\bar{u}_0,\bar{v}_0}( (U_t, V_t) \in \cdot) 
     - \pi_2 \|_{{\rm TV}} \leq f(n,t) ,
  \end{equation*}
  and by \eqref{eq.firstbit},
  \begin{equation*}
    d_n(\theta_0n + t) \leq f(n,t) + O(n^{-1}) .
  \end{equation*}
  Recall that $t_n = [2(1-\beta)]^{-1}(n\log n)$.
  For any $\gamma > 0$, let $t_n(\gamma) \deq t_n + \gamma n$.
  The theorem will be proved if we can establish a
  uniform coupling bound $f(n,t)$ such that
  \begin{equation*}
    \lim_{\gamma \rightarrow \infty} \limsup_{n \rightarrow \infty}
    f(n, t_n(\gamma)) = 0 .
  \end{equation*}

  Fix $(\bar{u}_0, \bar{v}_0) \in \Lambda_0$ and arbitrary
  $(\tilde{u}, \tilde{v})$.   Let $\sigma_0$ be
  any configuration with $(U(\sigma_0), V(\sigma_0))
  = (\bar{u}_0, \bar{v}_0)$, and let $\tilde{\sigma}$ be 
  any configuration with $(U(\tilde{\sigma}),
  V(\tilde{\sigma})) = (\tilde{u}, \tilde{v})$.
  We will construct, in two phases, a coupling $(X_t, \tilde{X}_t)$
  of the full Glauber dynamics with initial states $X_0 = \sigma_0$
  and $\tilde{X}_0 = \tilde{\sigma}$.  Given such
  a coupling, the projections 
  \begin{equation*}
    (U_t, V_t) \deq (U(X_t), V(X_t)), \quad \text{and} \quad
    (\tilde{U}_t, \tilde{V}_t) \deq (U(\tilde{X}_t), V(\tilde{X}_t))
  \end{equation*}
  are a coupling of the two-coordinate
  chains, started from $(\bar{u}_0, \bar{v}_0)$
  and $(\tilde{u}, \tilde{v})$.

  The magnetization coupling phase, lasting from time $0$ to time
  $t_n(\gamma)$ will ensure that $S_{t_n(\gamma)} =
  \tilde{S}_{t_n(\gamma)}$ with high probability, and that 
  \begin{equation*}
    \E_{\sigma_0,\tilde{\sigma}}\left[
      |\tilde{U}_{t_n(\gamma)} - U_{t_n(\gamma)}| \right]  
      = O(\sqrt{n}) .
  \end{equation*}

  During the two-coordinate coupling phase, 
  from time $t_n (\gamma)$ to time $t_n(2\gamma)$,
  with high probability the chains $(U_t)$ and
  $(\tilde{U}_t)$ coalesce.
  To facilitate coalescence, we must ensure that throughout the
  second phase with high probability 
  $X_t \in \Xi_1$ and $\tX_t \in \Xi_1$, where $\Xi_1$
  is as defined in \eqref{eq.g2def}.  Also, the coupling
  will ensure $S_t = \tilde{S}_{t}$ for
  all $t \in [t_n(\gamma),t_n(2\gamma)]$. 

  \emph{(i) Magnetization coupling.}
  Recall that $\taumag$, defined in \eqref{eq.taumagdef}, is
  the first time the normalized magnetizations agree.
  Let $H_1 \deq \{\taumag \le t_n (\gamma)\}$ be the
  event that the magnetizations couple by time $t_n(\gamma)$.
  By Lemma \ref{lem.magcouple}, there exists a constant $c$ 
  not depending on $\sigma_0$ or $\tilde{\sigma}$ such that
  \begin{equation*}
    \P_{\sigma_0, \tilde{\sigma}}\left( H_1^c \right) 
      \le c\cn^{-1/2}.
  \end{equation*}

  \emph{(ii) Two-coordinate chain coupling phase.}
  Assume that $\tilde{U}_{t_n} > U_{t_n}$; if this is not
  the case, just reverse the roles of $X_t$ and $\tilde{X}_t$ in what
  follows.   On the event $H_1$, 
  for $t \ge t_n (\gamma)$, use the coupling constructed in
  Lemma \ref{lem.uvcoupling}.  On the event $H_1^c$, 
  we let the two chains run independently for $t \ge t_n (\gamma)$.

  The outline of the remainder of the proof is
  as follows:  By \eqref{eq.rdrift}, the drift of the
  difference $\tilde{U}_t - U_t$ is non-positive, so
  it can be dominated by a process with independent and
  unbiased increments with values in $\{-1,0,1\}$, until
  $\tilde{U}_t - U_t$ hits zero. 
  Provided that the increments of $\tilde{U}_t - U_t$
  are non-zero with probability bounded away from $0$
  uniformly in $n$,
  the dominated process can be taken to be an unbiased
  random walk.    We will establish that
  at time $t_n(\gamma)$, the beginning of the second
  coupling phase, the expected difference 
  $\E_{\sigma_0, \tilde{\sigma}}[\tilde{U}_{t_n(\gamma)} - 
    U_{t_n(\gamma)}]$ is
  order $\sqrt{n}$.  Thus by comparison with random walk,
  the two-coordinate process will couple in  $O(n)$
  more steps. 

  We begin by showing that, if 
  $H_2(t) \deq \{X_t \in \Xi_1, \, \tilde{X}_t \in \Xi_1\}$,
  then
  \begin{equation} \label{eq.G1goal}
    \P_{\sigma_0, \tilde{\sigma}}\left( 
       \bigcup_{t_n(\gamma) \leq t \leq t_n(2\gamma)} H_2(t)^c 
       \right)
    = O(n^{-1}) .
  \end{equation}
  (Note that the bound above depends on $\gamma$.  This
  does not pose a problem, because the limit in $n$ is
  taken before the limit in $\gamma$ in \eqref{eq.hitempub}.)

  Recall the definition of $M_t(A)$ in \eqref{eq.MAdefn}.
  We introduce the following definitions:
  \begin{align*}
    A_0 & \deq \{i \st \sigma_0(i) = 1\}, \\
    B^\star 
      & \deq \bigcup_{t \in [t_n + \gamma n,\, t_n+ 2\gamma n]}
          \left\{ |M_t(A_0)| \ge n/32 \right\}, \\
    Y 
      & \deq \sum_{t \in [t_n + \gamma n, \, t_n + 2 \gamma n]}
            \one\{ |M_t(A_0)| > n/64 \} .
  \end{align*}
  (Note that $|A_0| = \bar{u}_0$.)
  Since $M_t(A_0)$ has increments in $\{-1,0,1\}$, if 
  $|M_{t_0}(A_0)| > n/32$, then $|M_t(A_0)| > n/64$ for all $t$ 
  in any interval of length $n/64$ containing $t_0$.  
  Consequently, $B^\star \subset \{Y > n/64\}$ and
  \begin{equation*}
    \P_{\sigma_0, \tsigma}(B^\star) 
      \leq \P_{\sigma_0, \tsigma}(Y > n/64) \
      \leq \frac{c_0 \E_{\sigma_0,\tsigma}[Y]}{n} .
  \end{equation*}
  By Lemma \ref{lem.spin-val}(ii), 
  $\P_{\sigma_0, \tsigma}(|M_t(A_0)| > n/64) = O(n^{-1})$
  for $t \geq t_n$, so $\E_{\sigma, \tsigma}[Y] = O(1)$ and
  \begin{equation*}
    \P_{\sigma_0, \tsigma}(B^\star) = O(n^{-1}).
  \end{equation*}
  Making analogous definitions and deductions for the
  chain $(\tilde{X}_t)$ shows that
  \begin{equation*}
    \P_{\sigma_0, \tilde{\sigma}}(\tilde{B}^\star) = O(n^{-1}) .
  \end{equation*}
  
  If $U_t \leq n/16$, then $\bar{u}_0 - U_t \geq 3n/16$, 
  since we are assuming that $\bar{u}_0 \geq n/4$.  Consequently,
  if $U_t \leq n/16$, then 
  \begin{equation*}
    |M_t(A_0)| = |U_t - (\bar{u}_0 - U_t)| 
      \geq (\bar{u}_0 - U_t) - U_t
      \geq \frac{n}{8} .
  \end{equation*}
  Similarly, $\bar{u}_0 - U_t \geq n/16$ implies that 
  $|M_t(A_0)| \geq 1/8$.   
  An analogous argument applied to $V_t$ and $\bar{v}_0 - V_t$
  shows that if either $V_t$ or $\bar{v}_0 - V_t$ does not exceed
  $n/16$, then $|M_t(A_0)| \geq n/8$, 
  since
  $|V_t - (\bar{v}_0 - V_t)| = |\bar{v}_0| \geq n/4$.  Finally,
  the same implications are obtained for the chains $(\tilde{X}_t),
  (\tilde{U}_t)$ and $(\tilde{V}_t)$.  
  To summarize, 
  \begin{equation*}
    H_2(t)^c \subset \{ |M_t(A_0)| \geq n/16 \} \cup 
      \{ |\tilde{M}_t(A_0)| \geq n/16 \}.
  \end{equation*}
  Thus,
  \begin{equation*}
     \P_{\sigma_0, \tilde{\sigma}}
     \left( \bigcup_{t_n(\gamma) \leq t \leq t_n(2\gamma)}
     H_2(t)^c \right) 
     \leq \P_{\sigma_0,\tsigma}(B^\star) 
       + \P_{\sigma_0,\tsigma}(\tilde{B}^\star) = O(n^{-1}) .
  \end{equation*}

  Recall that $R_t = |\tilde{U}_t - U_t|$, and let
  $H_2 \deq \bigcap_{t_n(\gamma) \leq t \leq t_n(2\gamma)} H_2(t)$.
  On the event $H_2$, the process $R_t$ can be dominated
  by a nearest-neighbor random walk, with delay, until
  the first time when $(R_t)$ visits $0$.  Then by 
  Lemma \ref{lem.rwest},  on $H_1$,
  \begin{equation*} 
    \P_{\sigma_0, \tilde{\sigma}}
      \left( \{\tau_c > t_n(2 \gamma)\} 
        \cap H_2 \given X_{t_n(\gamma)}, \tilde{X}_{t_n(\gamma)}
	\right)
    \leq 
    \frac{c_1 |R_{t_n(\gamma)}|}{
      \sqrt{n \gamma}} .
  \end{equation*}
  Taking expectation gives
  \begin{equation} \label{eq.rpwithe}
    \P_{\sigma_0, \tilde{\sigma}}
    \left( \{\tau_c > t_n(2 \gamma)\} \cap H_2 \cap H_1 \right)
    \leq \frac{c_1 \E_{\sigma_0,\tilde{\sigma}}[|R_{t_n(\gamma)}|]}{
      \sqrt{n \gamma} } .
  \end{equation}
  Observe that 
  \begin{equation*}
    U_t = M_t(A_0) + \bar{u}_0/2, \quad \text{and} \quad
    \tilde{U}_t  = \tilde{M}_t(A_0) + \bar{u}_0/2,
  \end{equation*}
  whence
  \begin{equation*}
    |U_t - \tilde{U}_t| = |M_t(A_0) - \tilde{M}_t(A_0)|
    \leq |M_t(A_0)| + |\tilde{M}_t(A_0)|.
  \end{equation*}
  Taking expectation shows that
  \begin{equation*}
    \E_{\sigma_0, \tilde{\sigma}}[ |R_t| ]
    \leq \E_{\sigma_0}[|M_t(A_0)|] + \E_{\tilde{\sigma}}[
      |M_t(A_0)|] .
  \end{equation*}
  Applying Lemma \ref{lem.spin-val}(iii) shows that
  $\E_{\sigma_0, \tilde{\sigma}}[ |R_{t_n(\gamma)}|]
  = O(\sqrt{n})$ .

  Using this estimate in \eqref{eq.rpwithe},
  we conclude that
  \begin{align*}
    \P_{\sigma_0, \tilde{\sigma}}\left( \tau_c > t_n(2\gamma) \right)
    & \leq \P_{\sigma_0, \tilde{\sigma}}\left( 
      \{\tau_c > t_n(2\gamma)\} \cap H_2 \cap H_1 \right)\\
    & \quad  + \P_{\sigma_0, \tilde{\sigma}}(H_2^c) 
      + \P_{\sigma_0, \tilde{\sigma}}(H_1^c)\\
    & \leq \frac{c_2}{\sqrt{\gamma}} + O(n^{-1}) .
  \end{align*}
  This gives the uniform coupling bound required.
\end{proof}

\subsection{Lower bound} \label{sec.cutoff-lower}

Recall $t_n = [2(1-\beta )^{-1}] n \log n$, and $\rho =
1-(1-\beta)/n$. Let us first restate the lower bound part of
Theorem~\ref{thm.hightemp}. 

\begin{theorem} \label{thm.hitemplb}
If $\beta < 1$, then
  \begin{equation*}
    \lim_{\cn \rightarrow  \infty}\liminf_{n \rightarrow \infty} 
     d_n\left( t_n - \cn n \right) = 1 .
  \end{equation*}
\end{theorem}
\begin{proof}
  It is enough to produce a suitable lower bound on the
  distance of the distribution of $S_t$ from its stationary
  distribution, since the chain $(S_t)$ is a projection of
  the chain $(X_t)$.

  Since $\theta_n(s) = O(n^{-2})$, expanding
  $\tanh[\beta(s+n^{-1})]$ around
  $\beta s$ in $f_n(s)$ and using equation \eqref{eq.deltaS1}
  shows that, for $s \geq 0$,
  \begin{equation} \label{eq.SrecurLBa}
    \E_{s_0}[S_{t+1} \mid S_t = s] 
      \geq \rho s - \frac{s^3}{2n} - O(n^{-2}) .
  \end{equation}
  By Remark \ref{rmk.symmetric},
  if $|S_t| > n^{-1}$, 
  \begin{equation} \label{eq.SrecurLB}
    \E_{s_0}\left[ |S_{t+1}| \given S_t \right]
      \geq \rho |S_t| - \frac{|S_t|^3}{2n} - O(n^{-2}) .
  \end{equation}
  This also clearly holds for $|S_t| = 0$ or $|S_t| = n^{-1}$.
  (In the latter case, $|S_{t+1}| \geq 1/n$.)

  Take the initial state $S_0$ to be $s_0 = s_0(\beta)$;
  we will specify the value of $s_0$ later.
  Define $Z_t \deq |S_t|\rho^{-t}$, whence
  $Z_0 = S_0 =s_0$.  Since $\rho^{-1} \leq 2$ for large $n$,
  from \eqref{eq.SrecurLB} it follows that
  \begin{equation*}
    \E_{s_0}[Z_{t+1} \mid Z_t] 
      \ge Z_t- \frac{\rho^{-t}[|S_t|^3+O(1/n)]}{n},
  \end{equation*}
  for $n$ large enough.
  Since $0 \leq |S_t| \leq 1$,
  \begin{equation} \label{eq.Zdrift}
    \E_{s_0}[Z_t - Z_{t+1} \mid Z_t] 
      \le \frac{\rho^{-t}[|S_t|^3+O(1/n)]}{n}
      \le \frac{\rho^{-t}[|S_t|^2+O(1/n)]}{n}.
  \end{equation}

  Applying Lemma \ref{lem.spin-val}(iii) with
  $A = \{1,2,\ldots,n\}$, we find that
  \begin{equation} \label{eq.absS}
   \E_{s_0}[|S_t|] \le |s_0|\rho^t + c_1 n^{-1/2}.
  \end{equation}
  Here and below, the constants $c_i$ depend only on
  $\beta$.

  Using the variance bound $\var(S_t) \le c_2 n^{-1}$
  (c.f.\ Proposition \ref{prop.magvar}) together with the
  inequality \eqref{eq.absS} shows that
  \begin{equation} \label{eq.second}
    \E_{s_0}\left[S_t^2\right] 
    = \left( \E_{s_0}[S_t] \right)^2 + \var(S_t) 
    \le s_0^2 \rho^{2t} + 2c_1 n^{-1/2}|s_0|\rho^t + c_3 n^{-1} \,.
  \end{equation}

  Taking expectations in \eqref{eq.Zdrift} and using \eqref{eq.second}
  yields
  \begin{equation*}
    \E_{s_0}[Z_t-Z_{t+1}] \le \frac{1}{n}
      \left[ s_0^2\rho^t + 2c_1n^{-1/2}|s_0| + c_3 \rho^{-t}/n\right] 
      + O(n^{-2}).
  \end{equation*}
  Let $t^\star = t_n - \alpha n/(1-\beta)$.
  Adding the increments $\E_{s_0}[Z_t]-\E_{s_0}[Z_{t+1}]$ for
  $t=0,\ldots, t^\star -1$, the above inequality 
  gives that
  \begin{equation*}
    s_0 - \E_{s_0}[Z_{t^\star}] 
      \le  \frac{s_0^2}{n(1-\rho)} 
      + \frac{2c_1|s_0| t^\star}{ n^{3/2}} + 
      c_3\frac{\rho^{-t^\star}}{n^2(1-\rho)} + O(t^\star n^{-2}).
  \end{equation*}
  Since $\rho^{-t^\star} \le n^{1/2}$, we deduce that
  \begin{equation} \label{eq.soZ}
    s_0-\E_{s_0}[Z_{t^\star}] 
    \le  \frac{s_0^2}{1-\beta} 
      +\frac{2c_2\log(n)}{n^{1/2}} + c_4n^{-1/2}.
  \end{equation}
  If $s_0 < (1-\beta)/3$ and $n$ is large enough,
  then the right-hand side of \eqref{eq.soZ}
  is less than $s_0/2$. 
  Thus 
  \begin{equation*}
    \E_{s_0}[|S_{t^\star}|] 
      \ge \frac{s_0 \rho^{t^\star}}{2} 
      \ge B:= \frac{s_0 e^\alpha}{2 n^{1/2}}.
  \end{equation*}
  By Proposition \ref{prop.magvar},
  $\max\{\var_{s_0}(S_t), \var_\mu(S)\} \leq c_5/n$.
  Thus 
  \begin{align*}
    B/2 &\leq \E_{s_0}[S_{t^\star}] 
      - \frac{s_0e^\alpha}{4c_5} \sqrt{\var_{s_0}(S_{t^\star})},\\
    B/2 &\geq \E_\mu[S]   
      + \frac{s_0e^\alpha}{4c_5} \sqrt{\var_{\mu}(S)} .
  \end{align*}
  Let $\pi_S$ be the stationary distribution of $(S_t)$, and
  let $A \deq [-B/2, B/2]$. Then
  \begin{equation*}
    \| \P_{s_0}( S_{t^\star} \in \cdot ) - \pi_S \|_{{\rm TV}}
    \geq \pi_S(A)  - \P_{s_0}( |S_{t^\star}| \in A ) 
    \geq 1 - 32c_5^2 e^{-2\alpha}/s_0^2 ,
  \end{equation*}
  where the last inequality follows from application of
  Chebyshev's inequality.
  The right-hand side clearly tends to $1$ as $\alpha
  \rightarrow \infty$.
\end{proof}

\section{Critical Case} \label{sec.critical}

In this section, we analyze the mixing time of the Glauber dynamics in
the critical case $\beta =1$, proving Theorem~\ref{thm.critical}. We
consider the upper and lower bounds separately.

\subsection{Upper bound}

\begin{theorem}
 If $\beta = 1$, then $\tmix = O(n^{3/2})$.
\end{theorem}
Recall the definition of $\tau_0$ in \eqref{eq.tau0defn}:
$\tau_0 \deq \min\{t \geq 0 \st |S_t| \leq 1/n \}$.

\begin{proof}
  We show that we can couple Glauber dynamics
  so that the magnetizations agree in order $n^{3/2}$ steps,
  and then appeal to Lemma \ref{lem.fromequalM} to show
  the configurations can be made to agree in another order
  $n\log n$ steps.

  \medskip
  \noindent
  \emph{Step 1:} Our first goal is to prove that 
  $\lim_{c \rightarrow \infty} 
    \P_\sigma( \tau_0 > c n^{3/2} ) = 0$,    
  uniformly in $n$.

  Recall the inequality \eqref{eq.ineq-2}:
  For $|S_t| > n^{-1}$,
  \begin{equation*}
    \E_\sigma\left[|S_{t+1}| \given S_t \right] 
    \leq \left(1 - \frac{1}{n} \right) |S_t|
         + \frac{1}{n}\tanh( |S_t| ) .
  \end{equation*}
  Multiply both sides above by $\one\{ \tau_0 > t \}$ and use
  the fact that $\tanh(0) = 0$ to find that
  \begin{equation*}
    \E_\sigma\left[|S_{t+1}| \one \{\tau_0 > t\} \given S_t \right]
      \leq \left(1 - \frac{1}{n} \right) |S_t|\one\{\tau_0 > t\}
        + \frac{1}{n}\tanh( |S_t|\one\{ \tau_0 > t\} ) .
  \end{equation*}
  Since $\one\{ \tau_0 > t+1 \} \leq \one\{ \tau_0 > t \}$,
  \begin{equation*}
    \E_\sigma\left[|S_{t+1}| \one \{\tau_0 > t+1\} \given S_t \right]
      \leq \left(1 - \frac{1}{n} \right) |S_t|\one\{\tau_0>t\}
        + \frac{1}{n}\tanh( |S_t|\one\{ \tau_0 > t\} ) .
  \end{equation*}
  Define $\splus_t \deq \E_{\sigma}[|S_t|\one\{\tau_0 > t\}]$.
  Take expectation above and apply Jensen's inequality to
  the concave function $\tanh$ restricted to the
  non-negative axis, to see that
  \begin{equation} \label{eq.srec}
    \splus_{t+1} \leq \left(1 - \frac{1}{n} \right) \splus_t 
       + \frac{1}{n}\tanh( \splus_t ).
  \end{equation}
  Thus, there exists a constant $c_\ep > 0$ such that, if 
  $\splus_t \geq \ep$, then 
  \begin{equation*}
    \splus_{t+1} - \splus_t  \leq - \frac{c_\ep}{n} . 
  \end{equation*}
  We conclude that there exists
  a time $t_\star = t_\star(n) = O(n)$ such that
  $\splus_t \le 1/4$ for all $t \ge t_\star$. 

  Expand $\tanh(x)$ in a Taylor series and
  use \eqref{eq.srec} to obtain
  \begin{equation*}
    \splus_{t+1} \leq \splus_t -\frac{(\splus_t)^3}{4n} + O(n^{-2}),
  \end{equation*}
  for $t \geq t_\star$.

  This shows that, for $n$ sufficiently large,
  $\splus_t$ is  decreasing for $t \geq t_\star$.  We will assume
  from now on that $n$ is large enough for this to hold.
  Given a decreasing sequence of numbers 
  \begin{equation*}
    1/4 \geq b_1 > b_2 > \cdots > 0,
  \end{equation*}
  let $u_i \deq \min\{t \geq t_\star \st \splus_t \leq b_i\}$.
  Since $\splus_t$ is decreasing,
  $b_{i+1} < \splus_t \leq b_i$ for all times $u_i \le t < u_{i+1}$. 
  Let $b_i = (1/4) 2^{-i}$.
  For $t \in (u_i, u_{i+1}]$, 
  \begin{equation*}
    \splus_{t+1} \leq \splus_t - \frac{b_i^3}{32n} + O(n^{-2}).
  \end{equation*} 
  It follows that 
  \begin{equation*}
    u_{i+1} - u_i 
    \leq \frac{16n}{b_i^2}\left[ 1 + O(b_i^{-3} n^{-1}) \right] 
  \end{equation*}
  Let $i_0 = \min\{ i \st b_i \leq n^{\alpha -1}\}$, where $\alpha$
  is a parameter to be chosen below.  If $\alpha > 2/3$, then
  $b_{i} \geq n^{-1/3 + \delta}$ for $i < i_0$, 
  for some $\delta > 0$.  In particular, $b_i^{-3}
  \leq n^{1 - \delta}$ and $O(b_i^{-3}n^{-1}) = o(n)$
  for $i < i_0$.  Thus for $n$ large enough,
  for $0 \leq i < i_0$,
  \begin{equation*}
    u_{i+1} - u_i \leq \frac{32n}{b_i^2} .
  \end{equation*}
  Summing the above, 
  \begin{align*}
    u_{i_0} - u_0 
      & \leq \sum_{i=0}^{i_0 - 1} \frac{32n}{b_i^2} 
        \leq \frac{c_0 n}{b_{i_0-1}^2} = O(n^{3-2\alpha}), \\
	\intertext{so}
    u_{i_0} & \leq O(n^{3-2\alpha}) + O(n),
  \end{align*}
  where the second inequality follows since $u_0 = t^\star = O(n)$.
  To summarize, provided $1 \geq \alpha > 2/3$, 
  there is a constant $c_1$ such that
  $\splus_t \leq n^{\alpha -1}$ for $t \geq c_1 n^{3-2\alpha}$.
  In particular, letting $r_n = c_1 n^{3-2\alpha}$, 
  there is a constant $c_2 > 0$ such that
  \begin{equation} \label{eq.ESrn}
    \E_{\sigma}\left[ |S^+_{r_n}| \one\{\tau_0 > r_n\}
     \right] \leq c_2 n^{\alpha - 1} .
  \end{equation}
  By the Markov property and Lemma \ref{lem.martrp},
  for some constant $c_3$,
  \begin{equation*}
    \P_\sigma( \tau_0 > r_n + \gamma n^{2\alpha} \mid X_{r_n})
    \leq \frac{c_3 n |S_{r_n}|}{ \sqrt{\gamma} n^\alpha } .
  \end{equation*}
  Multiplying both sides by $\one\{ \tau_0 > r_n \}$,
  taking expectation, and then using \eqref{eq.ESrn} shows that
  \begin{equation*}
    \P_\sigma( \tau_0 > r_n + \gamma n^{2\alpha}) 
    =  O(\gamma^{-1/2}) .
  \end{equation*}
  Choosing $\alpha = 3/4 > 2/3$, we see that
  \begin{equation*}
    \P_\sigma( \tau_0 > (c_1 + \gamma) n^{3/2} )
    = O(\gamma^{-1/2}) .
  \end{equation*}

  \medskip \noindent
  \emph{Step 2: Construction of coupling}.
  We now describe how to build a Markovian
  coupling $(X_t, \tilde{X}_t)$ of the Glauber dynamics
  such that the following holds:  There
  are constants $c_1 > 0$ and $b < 1$ such that, if $\taumag$ is as
  defined in \eqref{eq.taumagdef}, then for \emph{any}
  two configurations $\sigma$ and $\tilde{\sigma}$,
  \begin{equation} \label{eq.coupwithprob}
    \P_{\sigma, \tilde{\sigma}}( \taumag > c_1 n^{3/2} )
    \leq b.
  \end{equation}
  This is sufficient, since we only desire to prove $\tmix = O(n^{3/2})$.

  Fix two configurations $\sigma$ and $\tilde{\sigma}$, and
  suppose without loss of generality that
  $|S(\sigma)| > |S(\tilde{\sigma})|$.
  Define the stopping time $\tabs$ to be the first time
  the two chains cross over one another, i.e.\ 
  \begin{equation*}
    \tabs \deq \min\{ t \geq 0 \st |S_t| \leq |\tilde{S}_t| \},
  \end{equation*}
  and let $G_1 \deq \{ |S_{\tabs + 1}| = |\tilde{S}_{\tabs + 1}| \}$
  be the event that the two chains meet one step after $\tabs$.
  There is a constant $c_4 > 0$, not depending on $n$, such that
  $\P_{\sigma, \tilde{\sigma}}\left( G_1 \right) \geq c_4$.
  
  On $G_1^c$, couple the two chains independently.
  On $G_1$, we divide into two cases:

  \medskip \noindent
  \emph{Case $S_{\tabs + 1} = \tilde{S}_{\tabs + 1}$.}
  If this situation occurs, then couple such that the
  magnetizations continue to agree.  To do so, if
  a site $I$ is selected to update $X_t$ with a spin 
  $\spin$,  then pick a site in $\tilde{X}_t$ at random from those 
  with the same spin as $X_t(I)$, and update this site also with spin
  $\spin$. 

  \medskip \noindent
  \emph{Case $S_{\tabs + 1} = -\tilde{S}_{\tabs + 1}$.}
  In this case, we use the \emph{reflection coupling}:
  Suppose state $I$ is selected to update $X_t$, and
  the spin used to update is $\spin$.  Then pick a
  site in $\tilde{X}_t$ at random from those with
  spin $-X_t(I)$, and update with spin $-\spin$.    
  In this case, the process $(S_t)$ and $(\tilde{S}_t)$
  will be reflections of one another for $t \geq \tabs$.

  If $n$ is even, in either situation the magnetizations
  agree at time $\tau_0$, so $\taumag \leq \tau_0$. 
  For even $n$, run the chains together after $\tau_0$.
  If $n$ is odd, at time $\tau_0$ run the chains
  independently of one another for a single step.
  
  By Step 1 of the proof, there exists a constants $c_\star$ 
  and $c_6 > 0$ such that,
  for all $\sigma$,
  \begin{equation} \label{eq.tau0bnd1}
    \P_\sigma( \tau_0 + 1 \leq c_\star n^{3/2}) 
      \geq c_6.
  \end{equation}
  Let $G_2 = \{ \tau_0 + 1 \leq c_\star n^{3/2} \}$.
  
  Let $G_3$ be the event that the two chains
  couple at time $\tau_0 + 1$.   There exists
  some $c_5>0$ not depending on $n$ such that
  $\P_{\sigma}(G_3 \mid G_1 \cap G_2) \geq c_5$. (If $n$ is even,
  this probability is one.)
    
  Then
  \begin{equation*}
  	\P_{\sigma}(G_1 \cap G_2 \cap G_3 )
	  \leq \P_\sigma( \tau_c \leq c_\star n^{3/2}).
  \end{equation*}
  The probability on the left is uniformly bounded
  away from zero, completing the proof.
\end{proof} 

\subsection{Lower bound}

\begin{theorem}
  Suppose $\beta = 1$.  There is a constant $C_1 > 0$ such
  that $\tmix \geq C_1 n^{3/2}$.
\end{theorem}

\begin{proof}
  It will suffice to prove a lower bound on the mixing time of the
  magnetization chain $(S_t)$.  
  
  As usual, $S$ denotes the normalized magnetization in
  equilibrium.  The sequence $n^{1/4} S$ converges to
  a non-trivial limit law as $n \rightarrow \infty$.
  (This is proved in Simon and Griffiths \ycite{sg};
  see also \ocite{ellis:eldsm}*{Theorem V.9.5}.)
  Take $A > 0$ such that
  \begin{equation} \label{eq.SCritstat}
    \mu\left( |S| \leq A n^{-1/4} \right) 
      \geq 3/4 .
  \end{equation}
  
  Take $s_0 = 2An^{-1/4}$.  Let $(\tilde{S}_t)$ be a chain with the
  same transition probabilities as $(S_t)$, except at $s_0$.  At 
  $s_0$, the $\tilde{S}$-chain remains at $s_0$ with probability equal to
  the probability that the $S$-chain either moves up or remains
  in place at $s_0$.   The two chains can be coupled so that
  $\tilde{S}_t \leq S_t$ when both are started from $s_0$.  In
  particular, for all $s$, the inequality $\P_{s_0}(S_t \leq s) \leq
  \P_{s_0}(\tilde{S}_t \leq s)$ holds.   

  Let $Z_t = \tilde{S}_0 - \tilde{S}_{t \wedge \tau}$, where
  $\tau \deq \min\{t \geq 0 \st \tilde{S}_t \leq An^{-1/4}\}$.
  Note that $(Z_t)$ is non-negative.

  We will now show that if $\F_t$ is the sigma-algebra
  generated by $Z_1, \ldots, Z_t$, then there is a constant
  $c_A$ so that
  \begin{equation} \label{eq.Zsq}
    \E_{s_0}[ Z_{t+1}^2 - Z_t^2 \mid \F_t] \leq \frac{c_A}{n^2} .
  \end{equation}
  The equation \eqref{eq.Zsq} is clearly satisfied when
  $Z_t = 0$.  On the event $\tilde{S}_t = s$, where
  $An^{-1/4} < s < s_0$, the conditional distribution
  of $\tilde{S}_{t+1}$ is the same as the conditional
  distribution of $S_{t+1}$ given $S_t = s$.  Thus
  \begin{equation} \label{eq.dSt}
    \E_{s_0}[\tilde{S}_{t+1} \mid \tilde{S}_t = s]
    = \E_{s_0}[ S_{t+1} \mid S_t = s] \geq s - c_0 \frac{s^3}{n},
  \end{equation}
  for a constant $c_0$.  The inequality is obtained by
  expanding $\tanh$ in \eqref{eq.deltaS1}.
  From \eqref{eq.dSt}, it follows that
  \begin{equation} \label{eq.dZ}
    \E_{s_0}[Z_{t+1} \mid \F_t] \leq Z_t + \frac{c_0}{n} \tilde{S}_t^3 .
  \end{equation}
  We decompose the conditional second moment of $Z_{t+1}$ as
  \begin{equation} \label{eq.zvardecomp}
    \E_{s_0}[Z_{t+1}^2 \mid \F_t] 
      = \var(Z_{t+1} \mid \F_t) + 
      \left( \E_{s_0}[Z_{t+1} \mid \F_t] \right)^2 .
  \end{equation}
  Since $|Z_{t+1} - Z_t| \leq 2/n$,
  \begin{equation} \label{eq.varz}
    \var(Z_{t+1} \mid \F_t) 
      = \var(Z_{t+1} - Z_{t} + Z_{t} \mid \F_t) 
      = \var(Z_{t+1} - Z_{t} \mid \F_t)
      \leq \frac{4}{n^2}.
  \end{equation}
  By \eqref{eq.dZ}, for $t < \tau$, there is a constant $c_1$
  (depending on $A$) so that
  \begin{equation} \label{eq.ez2}
    \E^2_{s_0}[Z_{t+1} \mid \F_t] \leq Z_t^2 + 2\frac{c_0}{n}Z_t
    \tilde{S}_t^3 + \frac{c_0^2 \tilde{S}_t^6}{n^2} 
    \leq Z_t^2 + c_1 n^{-2} .
  \end{equation}
  Using the bounds \eqref{eq.varz} and \eqref{eq.ez2} in
  \eqref{eq.zvardecomp} establishes \eqref{eq.Zsq}.
    We conclude that
  \begin{equation} \label{eq.zsm}
    \E_{s_0}[Z_t^2] \leq c_A n^{-2} t .
  \end{equation} 
  Note that
  \begin{equation*}
    \E_{s_0} [ Z_t^2 ] \geq \E_{s_0}[ Z_t^2 \one\{ \tau \leq t \} ]
    \geq \frac{A^2}{n^{1/2}} \P_{s_0}( \tau \leq t ),
  \end{equation*}
  which together with \eqref{eq.zsm} shows that
  \begin{equation*}
    \P_{s_0}( \tau \leq t ) \leq \frac{c_A t}{A^2 n^{3/2}} .
  \end{equation*}
  Taking $t = (A^2/4c_A) n^{3/2}$ above shows that
  \begin{equation*}
    \P_{s_0}( S_t \leq A n^{-1/4} ) \leq \frac{1}{4}.
  \end{equation*}
  This, together with the bound
  \eqref{eq.SCritstat}, proves that
  $d(c_3 n^{3/2}) \geq 1/2$,
  where $c_3 = A^2/4c_A$.  That is, $\tmix \geq c_3n^{3/2}$.
\end{proof}

\section{Truncated Dynamics for Low Temperature} \label{sec.lowtemp}

We now consider the case $\beta > 1$. As stated in the introduction,
the mixing time for the full Glauber dynamics is exponential in 
$n$.   This is proved via an upper bound on the Cheeger constant,
defined as
\begin{equation*}
  \Phi \deq \min_{A \st \mu(A) \leq 1/2} \frac{\sum_{x \in A, y \not\in A}
   \mu(x)P(x,y)}{\mu(A)} ,
\end{equation*}
where $P$ is the transition matrix for the Glauber dynamics.
By taking $A = \{ \sigma \st \mu(\sigma) \geq 0 \}$ and estimating
$\left[\sum_{x \in A, x \not\in A}\mu(x)P(x,y)\right]/\mu(A)$, 
when $\beta > 1$ there are positive constants $c_1$ and $c_2$ such
that $\Phi \leq c_1e^{-c_2 n}$. The spectral gap of $\Pm$ is bounded
below by $c_3/\Phi$ (see, for example, \ocite{S:ARGC}.)  The mixing
time, in turn, is bounded below by the spectral gap  (see, for
example, \ocite{AF:MC}.) The details of this standard argument can be
found in the forthcoming book \fullocite{lpw}.  That the Glauber dynamics
is slow mixing for $\beta > 1$ was understood as far back as
\fullocite{GWL}, although they lacked the tool of the Cheeger inequality
to make a complete proof.

Here we study the Glauber dynamics confined to the
configurations where the magnetization is non-negative, and
show that the restricted Glauber dynamics has a mixing time of order
$n \log n$. 

We remind the reader of the exact mechanism for restricting the
dynamics.   The usual dynamics are run from a state with non-negative
magnetization.  If a move to a state $\eta$ is proposed, and
$\eta$ has negative magnetization, then the chain moves to
$-\eta$ instead.  

To establish an $O(n \log n)$ upper bound on the mixing time, we need
to estimate the hitting times of the normalized magnetization chain. 

\begin{lemma} \label{lem.hittop}
  Let $\beta > 1$.  Let $s^\star$ denote the
  unique positive solution to $\tanh(\beta s) = s$,
  and for $\alpha > 0$ define
  \begin{equation} \label{eq.taustardefn}
    \tau^\star = \tau^\star(\alpha)
      \deq \inf\{ t \geq 0 : S^+_t \leq s^\star 
        + \alpha n^{-1/2} \} .
  \end{equation}
  There exists a constant $c > 0$, depending on $\alpha$
  and $\beta$, such that
  \begin{equation*}
    \lim_{n \to \infty }
      \P_\sigma( \tau^\star > c n \log n )  = 0 .
  \end{equation*}
\end{lemma}
\begin{proof}
  Let $\gamma^\star \deq \beta \cosh^{-2}(\beta s^\star)$.
  First, we show that
  \begin{equation} \label{eq.lotempcon1}
    \E_\sigma[ S^+_{t+1} - s^\star \mid S^+_t =s ]
      \leq \left[1 - \frac{(1-\gamma^\star)}{n} \right](s - s^\star) .
  \end{equation}
  By Remark \ref{rmk.truncated} and \eqref{eq.ineq-2}, 
  for $S^+_t > 1/n$ 
  \begin{equation*} 
    \E_{\sigma}\left[ S^+_{t+1} - S^+_t \given S^+_t \right]
      \leq  \frac{1}{n}\left[ \tanh(\beta S^+_t) - S^+_t \right] . 
  \end{equation*}  
  Since $\beta > 1$, it follows that
  $\gamma^\star = \beta \cosh^{-2}(\beta s^\star) < 1$. 
  By the mean-value theorem, for $y > 0$,
  \begin{equation*}
    \tanh[ \beta(s^\star + y) ] - \tanh(\beta s^\star)
    = \frac{\beta}{\cosh^2(\bar{s})} y ,
  \end{equation*}
  for some $\bar{s} \in [s^\star, s^\star + y]$.  
  Since $\cosh(x)$ is increasing for $x \geq 0$,
  the right-hand side is bounded above by
  $\gamma^\star y$.  Thus, for $y \geq 0$,
  \begin{equation} \label{eq.tanhdif}
    \tanh[ \beta(s^\star + y)] \leq s^\star + \gamma^\star y .
  \end{equation}
  Hence,
  \begin{equation*}
    \E_\sigma[S^+_{t+1} - S^+_t \mid S_t^+ = s] 
      \leq -(s-s^\star) \frac{(1-\gamma^\star)}{n} ,
  \end{equation*}
  from which \eqref{eq.lotempcon1} follows.

  By \eqref{eq.lotempcon1},
  \begin{equation*}
    Y_t 
      \deq \left[1- \frac{(1-\gamma^\star)}{n}\right]^{-t}
        (S^+_t-s^\star)
  \end{equation*}
  defines a non-negative supermartingale for $t < \tau^\star$.
  By optional stopping,
  \begin{multline*}
    1  \geq \E_\sigma[Y_{\tau^\star \wedge t}] 
       \geq \E_\sigma\left[ (1-(1-\gamma^\star)/n)^{-t 
	   \wedge \tau^\star} 
         (S^+_{\tau^\star \wedge t}-s^\star)\right] \\
       \geq c_1 n^{-1/2} [1-(1-\gamma^\star)/n]^{-t} 
         \P_\sigma( \tau^\star > t ) .
  \end{multline*}
  Hence 
  $\P_\sigma( \tau^\star > t ) \leq c_1 n^{-1/2}[1-(1-\gamma)/n]^t$, 
  and the lemma is proved.
\end{proof}

\begin{proposition} \label{prop.hitfromzero}
  Let $\beta > 1$. For $c_3 > 0$, if 
  \begin{equation*}
    \tau_\star = \tau_\star(c_3) 
      \deq \min\{t \geq 0 \st 
        S^+_t \geq s^\star + c_3n^{-1/2} \},
  \end{equation*}
  then
  \begin{equation} \label{eq.hitfromzero}
    \E_0[ \tau_{\star}]  = O(n \log n) . 
  \end{equation}
\end{proposition}

Proposition~\ref{prop.hitfromzero} is proved in Section~\ref{sec.hit}.
Meanwhile, we state and prove Theorem~\ref{thm.lotempub} below, which
establishes the upper bound. 

\begin{theorem} \label{thm.lotempub}
  Let $\beta > 1$. There is a constant $c(\beta)$ so
  that $\tmix(n) \leq c(\beta) n \log n$ for the Glauber
  dynamics restricted to $\X^+$.
\end{theorem}
  
\begin{proof}
  We show that there is a coupling $(X_t^+, \tX^+_t)$
  of the restricted Glauber dynamics started from
  states $\sigma$ and $\tsigma$ such that,
  if $\taumag$ is the first time $t$ with
  $S_t^+ = \tilde{S}_t^+$, then
  \begin{equation*}
    \limsup_{n \rightarrow \infty}
      \P_{\sigma,\tilde{\sigma}}( \taumag > cn\log n )
	\rightarrow 0
      \quad \text{as } c \rightarrow \infty .
  \end{equation*}
  An application of Lemma \ref{lem.fromequalM} will then 
  complete the proof.

  By monotonicity, it is enough to consider the 
  the starting
  positions $0$ and $1$.  The ``top'' chain with starting
  position $1$ we denote by $(S^T_t)$, and the ``bottom''
  chain with starting position $0$ we denote by
  $(S^B_t)$.   Let $\mu^+$ be the stationary distribution of the
  restricted magnetization chain, and let $(S_t)$ be a 
  stationary copy of the restricted magnetization chain,
  that is, started with initial distribution $\mu^+$.   

  Initially, all the chains are independent of one another.
  Given constants $c_1 \leq c_2$, let
  \begin{align*}
    \tau_1 & = \min\{ t \geq 0 \st S^T_t 
                        \leq s^\star + c_1n^{-1/2}\},\\
    \tau_2 & = \min\{ t \geq 0 \st S^B_t 
                        \geq s^\star + c_2n^{-1/2}\}.
  \end{align*}
  Suppose that $\tau_1 \leq \tau_2$.  On the event 
  $S_{\tau_1} \geq s^\star + c_1 n^{-1/2}$, for $t \ge \tau_1$ we
  couple together monotonically the $S$-chain and the
  $S^T$-chain (that is, such that $S_t \ge S^T_t$ for all $t \ge
  \tau_1$), and continue to evolve the $S^B$-chain independently of
  $S_t$ and $S^T_t$.   On the event $S_{\tau_1} < s^\star + c_1
  n^{-1/2}$, we continue to  run all three chains independently. Then
  at time $\tau_2$,  on the event that $S_{\tau_2} \leq s^\star + c_2
  n^{-1/2}$,  couple together all three chains monotonically (so that
  $S^T_t \le S_t \le S_t^B$ for all $t \ge \tau_2$).  If $S_{\tau_2}
  > s^\star + c_2$, just let the chains run independently.
  The case $\tau_2 < \tau_1$ is handled analogously.   
  
  Note that, since $(S_t)$ is independent of $(S^T_t)$ until after
  time $\tau_1$, the random variable $S_{\tau_1}$ is independent of
  $\tau_1$ and hence still stationary.

  Let $c_3>0$ be a constant, and define events $H_1, H_2$ by
  \begin{align*}
    H_1 & = \{ \tau_1 \le c_3 n \log n  \} 
      \cap \{S_{\tau_1} \ge s^\star + c_1 n^{-1/2} \} ,\\
    H_2 & = \{ \tau_2 \le c_3 n \log n  \}  
      \cap \{ S_{\tau_2} \le s^\star + c_2 n^{-1/2}. \}.
  \end{align*} 
  Then
  \begin{equation} \label{eq.g1p}
    \P_{\sigma, \tsigma}(H_1^c)
      \leq \P_{\sigma,\tsigma} ( \tau_1 > c_3 n\log n ) 
        + \mu^+(0, s^\star + c_1n^{-1/2}),
  \end{equation}
  and
  \begin{equation} \label{eq.g2p}
    \P_{\sigma, \tsigma}(H_2^c)
      \leq \P_{\sigma, \tsigma}(\tau_2 > c_3 n\log n ) 
        + \mu^+(s^\star + c_2 n^{-1/2},1).
  \end{equation}
  Now observe that on the event $H_1 \cap H_2$ the chains
  $(S^T_t)$ and $(S^B_t)$ have crossed over by the time
  $c_3 n \log n$, and that by \eqref{eq.g1p} and 
  \eqref{eq.g2p},
  \begin{equation*}
    \P_{\sigma, \tsigma}( H_1 \cap H_2 ) 
      \geq 1 - \P_{\sigma, \tsigma}( \tau_1 > c_3n\log n)
      - \P_{\sigma, \tsigma}( \tau_2 > c_3n\log n) - \mu^+(I^c),
  \end{equation*}
  where $I = (s^\star + c_1n^{-1/2}, s^\star + c_2 n^{-1/2})$.
  
  Since, as a consequence of Theorem 2.4 of Ellis, Newman,
  and Rosen~\ycite{ENR},
  the stationary magnetization satisfies a central
  limit theorem, $\mu^+(I^c) < 1$ uniformly
  in $n$. Further, 
  \begin{equation*}
    \lim_{n \rightarrow \infty} 
      \P_{\sigma, \tsigma} (\tau_1 > c_3n\log n) = 0 
    \quad \text{and} \quad
    \lim_{n \rightarrow \infty}
      \P_{\sigma, \tsigma} (\tau_2 > c_3n\log n) = 0, 
  \end{equation*}
  by Lemma \ref{lem.hittop} and Proposition
  \ref{prop.hitfromzero}, respectively.  Hence the
  probability that $S^T$ and $S^B$ will have crossed by the time $c_3 
  n \log n$ stays bounded away from 0 as $n \to \infty$.  
  
  Finally, observe that, whenever the two chains cross, they coalesce
  with probability bounded away from 0 uniformly in $n$, which
  completes the proof.  
\end{proof}

\subsection{Hitting times for birth-and-death chains}

\label{subs.birth-death}

A \emph{birth-and-death chain} on $\{0,1,\ldots,N\}$ is a Markov chain
$(Z_t)$  on ${\mathbb Z}^+$ with transitions $Z_{t+1} - Z_t$ contained
in the set $\{-1, 0, 1\}$. 

This section contains a few standard results concerning the hitting
times of birth-and-death chains. We shall use these in the proof of
Proposition~\ref{prop.hitfromzero} in the next section.

Define
\begin{align*}
  p_k & = \P(Z_{t+1} - Z_t = +1 \mid Z_t = k) & k=0,1,\ldots,N-1, \\ 
  q_k & = \P(Z_{t+1} - Z_t = -1 \mid Z_t = k) &k=1,\ldots, N, \\
  r_k & = \P(Z_{t+1} - Z_t = 0  \mid Z_t = k) &k=0,\ldots,N.
\end{align*}
Clearly, $p_k + q_k + r_k = 1$ for all $k$ if we define $q_0 = p_N = 0$.
Using $\pi$ to denote the stationary distribution of the chain, we have
\begin{align*}
  \pi(1) & = C_{p,q,r}, \\
  \pi(k) & = C_{p,q,r} \prod_{j = 1}^k \frac{p_{j-1}}{q_j},
    & k = 1,\ldots,N,
\end{align*}
where $C_{p,q,r} = [1 + \sum_{k=1}^n p_{j-1}q_j^{-1}]^{-1}$ is
a normalizing constant.

Now, let $\ell < N$ be a positive integer, and let $Z_t^{(\ell)}$ be
a restriction of $Z_t$ to the set $\{0,\ldots, \ell\}$.  In other
words, when at $k \in \{0,\ldots, \ell-1\}$, the chain makes
transitions from $k$ as the original chain, but when at $\ell$, it moves
to $\ell-1$ with probability $q_{\ell}$ and stays at $\ell$ with
probability $p_{\ell} + r_{\ell}$.
Let $\pi^{(\ell)}$ be the stationary measure of $Z_t^{(\ell)}$.  It is
easy to verify that there is a constant 
$C_{p,q,r}^\ell$ such that 
\begin{equation*}
  \pi^{(\ell)}(k) = C^\ell_{p,q,r} \pi (k)  \quad 
  \text{for } k=0, 1, \ldots, \ell.
\end{equation*}
In other words, under the stationary measure of the restricted chain,
the states $0,1,\ldots, k$ each have the same relative weights as in
the unrestricted chain. 

For $k \in \{0,1, \ldots, N\}$ let
\begin{align*}
  \tau_k   & = \inf\{t \geq 0\st Z_t=k\}, \\
  \tau^+_k & = \inf \{t > 0 \st Z_t =k \}.
\end{align*}
Then (see for instance \fullocite{lpw}) for $k=0,1, \ldots, N-1$,
\begin{equation} \label{eq.bdreturn}
  \frac{1}{\pi^{(\ell)}(\ell)} 
  = \E_{\ell}^{(\ell)}[ \tau_{\ell}^+ ] 
  = 1 + q_{\ell}\E_{\ell-1}(\tau_{\ell}).
\end{equation}
In the above, $\E_j$ and $\E^{\ell}_j$ respectively denote the
expectation operators corresponding to the unrestricted and restricted
chain starting in $j$. We shall now apply identity~(\ref{eq.bdreturn})
to the Glauber dynamics magnetization chain. 

\subsection{Hitting time for magnetization} \label{sec.hit}

\begin{proof}[Proof of Proposition \ref{prop.hitfromzero}]
  Here it is more convenient to work with $M_t = nS(X^+_t)/2$, which
  is a birth-and-death chain with values in $\{0, \ldots, n/2 -1,
  n/2\}$. Note that, if $n$ is odd, this chain is not integer-valued,
  but this causes no difficulties, as one can simply shift all states
  by -1/2. 

  Let $\ell^\star = \floor{ns^\star}$. Let $c > 0$ be a
  constant. Also, throughout the calculation, $C$ will denote a
  generic positive constant whose value may be adjusted between
  inequalities. In the notation of Section~\ref{subs.birth-death}, we
  have for  $\ell \in \{1,\ldots, \lceil n s^\star + cn^{1/2}\rceil \}$, 
  \begin{equation*}
    \E_{\ell-1}[\tau_{\ell}] 
      \le \frac{1}{q_{\ell} \pi^{(\ell)}(\ell)}. 
  \end{equation*}
  The probability of moving left, $q_{\ell}$, is bounded away from 0, 
  uniformly in $\ell \in \{1, \ldots, n/2\}$.
  Consequently, writing $\ell = n x$ and $j =ny$, 
  we obtain the upper bound
  \begin{equation*}
    \E_{\ell-1}[\tau_{\ell}]
      \le C \frac{ \sum_{j=0}^{\ell} \binom{n}{n/2 + ny}
                   \exp\left(\beta 2ny^2\right) 
      }{ 
        \binom{n}{n/2+ nx} 
	\exp\left( 2\beta n x^2 \right)
      }.
  \end{equation*}
  Applying Stirling's formula, the right-hand side is bounded
  above by
  \begin{equation*}
    C \frac{
      \sum_{j=0}^{\ell}(1+y)^{-(1+2y)n/2}(1-2y)^{-(1-2y)n/2}
        (1-4y^2)^{-1/2} \exp\left( 2\beta n y^2 \right) 
    }{ 
      (1+ 2x)^{-(1+ 2x)n/2} (1-2x)^{-(1-2x)n/2} 
      (1-4x^2)^{-1/2} \exp\left( 2\beta n x^2 \right)
    },
  \end{equation*}
  which can be rewritten as
  \begin{multline*}
    C \frac{
      \sum_{j=0}^{\ell} \exp\left[ -n f(y) \right] 
        (1-4y^2)^{-1/2}
    }{
      \exp\left[ -n f(x) \right] (1-4x^2)^{-1/2}
    } \\
    = 
    C \sum_{j=0}^{\ell} \exp\left[n (f(x)- f(y) \right] 
      \left( \frac{1-4x^2}{1-4y^2} \right)^{1/2},
  \end{multline*}
  where 
  \begin{equation*}
    f(z) = \frac{1}{2}(1+2z) \log(1+2z) 
      + \frac{1}{2} (1-2z) \log(1-2z) - 2 \beta z^2.
  \end{equation*}

  Since $\ell/n \le (\ell^\star+O(\sqrt{n}))/n < 1$ uniformly in $n$,  
  we can bound  
  \begin{equation*}
    \sup_n \sup_{0 \le y \le s^\star=\ell^\star/n} 
      \left( \frac{1-4x^2}{1-4y^2} \right)^{1/2} 
      \le C.
  \end{equation*}
  It follows that the behavior of each term in the sum is dominated by
  the behavior of the exponential factor 
  $\exp\left[n (f(x)-f(y)) \right]$, and so 
  it is enough to upper bound the expression
  \begin{equation*}
    \sum_{j=0}^{\ell} \exp \left[n (f(x)- f(y)) \right].
  \end{equation*}
  We then need to look for stationary points of $f$ in the interval
  $[0,1]$; we have 
  \begin{align*}
    f'(z)  & =  \log (1+2z) -  \log(1-2z) -4\beta z\\
    f''(z) & =  \frac{1}{1-4z^2} - 4\beta,
  \end{align*}
  so $f'(z)=0$ if and only if
  \begin{equation} \label{eq.fixed}
    \frac{1+2z}{1-2z} = e^{4\beta z},
  \end{equation}
  or, equivalently,
  \begin{equation*}
    2z = \tanh(2\beta z).
  \end{equation*}
  When $\beta < 1$, the unique maximum of $f$ is at $x=0$.
  When $\beta >1$, there is a local maximum of $f$ at $s=0$,
  and as mentioned earlier, there is a unique $0 < s^\star < 1$ 
  minimizing $f$. As before, we write 
  $\ell^* =\lfloor ns^\star \rfloor$. 

  By the above, when $x < s^\star$, 
  \begin{equation*}
    \E_{\ell-1}[\tau_{\ell}] 
      \le C \sum_{j=0}^{\ell} \exp\left[n(f(x)-f(y))\right],
  \end{equation*}
  and $f(x) \le f(y)$ for all $y \le x$. 

  Throughout the calculation below, we shall use the fact that 
  $f'(y) < 0$ for all $y \in [0,s^\star)$, and that the second
  derivative $f''(y)$ exists and is uniformly bounded in that range, as
  $s^\star < 1/2$.  

  Suppose $x = O(n^{-1/2})$, i.e. $\ell = O(\sqrt{n})$. 
  Then
  \begin{align*}
    \E_{\ell-1}[\tau_{\ell}] 
      & \le C \sum_{j=0}^{\ell} \exp\left[2 f'(x)(nx -ny)+
        O(n(x-y)^2)\right] \\
      & \le  C \sum_{j=0}^{\ell } \exp\left[(f'(\ell /n) 
          (\ell - j)\right]  \\
      & \le \sqrt{n}\left[1 + O(n^{-1/2})\right],
  \end{align*}
  valid for $1 \le \ell \le C_1 \sqrt{n}$.  The final bound is valid
  as  $f'(\ell/n ) < 0$, and so each term is bounded by a constant. 

  Similarly (taking $C_1=20$) we have, for 
  $20 \sqrt{n} \le \ell \le \ell^\star/2$,
  \begin{align*}
    \E_{\ell-1}[\tau_{\ell}] 
      & \le  C \sum_{j=0}^{\ell} \exp\left[ f'(c_{\ell,y})
        (\ell - j) + O(n(x-y)^2) \right]\\
      & \le  C \sum_{j=0}^{\ell} \exp\left[ f'(c_{\ell,y})
        (\ell - j ) \right], 
  \end{align*}
  where $c_{\ell,y}$ is between $x$ and $y$ (we could take
  $c_{\ell,y}=x$, for each $y$, by the uniform boundedness of the
  second derivative). There exists a constant $c_1>0$ such that, 
  if $j \ge \ell /2$, then $f'(c_{\ell,y}) \le - c_1 \ell/n$. 
  Then there exists a constant $c_2 >0$ such that, for 
  $j \le \ell/2$,
  \begin{equation*}
    f(j/n)-f(\ell/n) \le -c_2.
  \end{equation*}
  This in turn implies that the sum of remaining terms is negligible.
  More precisely,
  \begin{equation*}
    \sum_{j=0}^{\ell/2} \exp\left[n (f(\ell/n)-f(j/n)) \right] 
      \le n\exp(-c_2 n).
  \end{equation*}
  It follows that
  \begin{align*}
    \E_{\ell-1}[\tau_{\ell}] 
      & \le \sum_{j=\lfloor \ell/2 \rfloor}^{\ell } 
              \exp\left[ - c_1 \ell n^{-1} (\ell - j ) \right] 
	      + n \exp(-c_2 n) \\
      & \le \frac{1}{1-\exp(-c_1\ell /n)} + n \exp(-c_2 n)\\
      & \le \frac{ Cn}{\ell},
  \end{align*}
  for some constant $C > 0$, uniformly in $n$.

  Now suppose that 
  $\ell^\star/2 \le \ell \le \ell^\star -20\sqrt{n}$. Then, 
  for some constant $\tilde{c}_1>0$, 
  $f'(c_{\ell,y}) \le -\tilde{c}_1 (\ell^\star-\ell)/n$, as 
  long as $j=yn \ge \ell/2$. Also, there exists a constant
  $\tilde{c}_2 >0$ such that, for $j \le \ell/2$, 
  \begin{equation*}
    f(j/n)-f(\ell/n) \le -\tilde{c}_2,
  \end{equation*}
  and so the contribution due to the terms with $j \le \ell/2$ is
  negligible.
  
  Then a calculation similar to that for $20 \sqrt{n} \le \ell \le
  \ell^*/2$ above implies that there is a constant $C > 0$
  such that
  \begin{equation*}
    \E_{\ell-1}[\tau_{\ell}]  \le  \frac{C n}{\ell^*-\ell},
  \end{equation*}
  uniformly in $n$. Similarly, if 
  $\ell^\star -20\sqrt{n} \le \ell \le \lceil n s^* 
  + c \sqrt{n}\rceil$, then we see that
  \begin{equation*}
    \E_{\ell-1}[\tau_{\ell}] = O(\sqrt{n}).
  \end{equation*}

  Summing over $\ell$, we obtain an upper bound on the expected
  hitting time of $\lceil n s^\star + c \sqrt{n} \rceil$ starting from
  0, as follows: 
  \begin{align*}
    \E_0[\tau_{\ell^\star + c \sqrt{n}}] 
      & = \sum_{\ell=0}^{\ell^\star + c \sqrt{n}} 
            \E_{\ell-1}[\tau_{\ell}] \\
      & \le  C \left(  \sqrt{n} \times \sqrt{n}
	    + \sum_{\ell=1}^n \frac{n}{\ell} +
	    \sum_{\ell=\ell^\star-1}^{\ell^\star/2} 
            \frac{n}{\ell^\star-\ell} \right)\\
      & \le  C (n + n\log n),
  \end{align*}
  where $C$ is once again a generic constant, and was changed to $2C$
  in the last inequality. 
\end{proof}

Related results on the magnetization chain can be found in
\ocite{OV}.

\subsection{Lower bound}

\begin{theorem} \label{thm.ltlb}
  Assume that $\beta > 1$.  For the Glauber dynamics restricted to
  configurations with non-negative magnetization,
  $\tmix(n) \geq (1/4) n\log n$. 
\end{theorem}
The Glauber dynamics restricted to configurations with non-negative
magnetization will be denoted by $(X^+_t)$.
\begin{proof}
  Recall again that $s^\star$ is the unique positive 
  solution to $\tanh(\beta s^\star) = s^\star$.

  Since we are proving a lower bound, it suffices to consider
  any specific starting state; we take $X_0^+$ to be the all
  plus configuration.   

  We let $(X_t^+, \tilde{X}^+_t)$ be the monotone coupling, where
  $X_0^+$ is the all plus configuration and $\tilde{X}^+_0$ has the
  stationary distribution $\mu^+$.   
  We write $\P_{\one,\mu^+}$ and $\E_{\one,\mu^+}$ for the probability
  measure and expectation operator on the space where
  $(X^+_t, \tilde{X}^+_t)$ is defined.  

  Let ${\mathcal B}(\sigma) \deq \{ i \st \sigma(i) = -1 \}$,
  and $B(\sigma) \deq |{\mathcal B}(\sigma)|$. 

  By the central limit theorem for the stationary magnetization,
  (c.f.\ Ellis, Newman, and Rosen \ycite{ENR}), for some $0 < c_1 < 1$,
  \begin{equation*}
    \P_{\one, \mu^+}\left( B(\tX_0^+) \leq c_1 n \right) = 
    \mu^+(\{ \sigma \st B(\sigma) \leq c_1 n \})
    = o(1) .
  \end{equation*}

  Let $N_t$ be the number of the sites in ${\mathcal B}(\tX_0^+)$
  which have not been updated by time $t$.
  By writing $N_t$ as a sum of indicators,
  \begin{equation*}
    \E_{\one, \mu^+}\left[N_t \,\big|\, B(\tX_0^+) \right]
    =  B(\tX_0^+)[1-n^{-1}]^t,
  \end{equation*}
  and so, for some $c_2 > 0$,
  \begin{equation*}
    \E_{\one, \mu^+}\left[N_{t^\star_n} \given B(\tX_0^+) \right] 
      \geq c_2 B(\tX_0^+) n^{-1/4},
  \end{equation*}
  where $t^\star_n = (1/4)n\log n$.  Also, since these
  indicators are negatively correlated, $\var_{\one,\mu^+}(N_t) \leq
  n$ for all $t$.  Applying Chebyshev's inequality shows that, for
  some $c_3 > 0$,  on the event $\{ B(\tX_0^+) > c_1 n \}$,
  \begin{equation*}
    \P_{\one,\mu^+}\left( N_{t^\star_n} 
      \leq c_3 n^{3/4} \given B(\tX_0^+) \right) = o(1) , 
  \end{equation*}
  where the $o(1)$ bound is uniform in $B$.
  We conclude that
  \begin{align*}
    \P_{\one,\mu^+}\left( N_{t^\star_n} \leq c_3 n^{3/4} \right)
      & \leq \P_{\one, \mu^+}\left( B(\tX_0^+) \leq c_1n \right) \\
      & \quad + \P_{\one, \mu^+}\left( N_{t^\star_n} \leq c_3n^{3/4}
      \text{ and } B(\tX_0^+) > c_1 n \right) \\
      & = o(1) .
  \end{align*}
  Suppose now that $N_{t^\star_n} > c_3 n^{3/4}$.
  It follows that 
  $S_{t^\star_n} \geq \tilde{S}_{t^\star_n} + c_4 n^{-1/4}$
  for some $c_4 > 0$.
  Thus, if $S_{t_n^\star} \leq s^\star + c_5 n^{-1/4}$ for a small
  constant $c_5 > 0$, then
  $\tilde{S}_{t_n^\star} \leq s^\star + (c_5-c_4)n^{-1/4}$.
  Therefore,
  \begin{align*}
    \P_{\one, \mu^+}\left(S_{t_n^\star} \leq s^\star + c_5 n^{-1/4}\right)
    & \leq o(1) 
       + \P_{\one, \mu^+}\left( N_{t_n^\star} > c_3 n^{3/4} \text{ and }
           S_{t_n^\star} \leq s^\star + c_5 n^{-1/4} \right) \\ 
    & \leq o(1) 
       + \P_{\one, \mu^+}\left( \tilde{S}_{t_n^\star} \leq s^\star 
             + (c_5 - c_4)n^{-1/4} \right) .
  \end{align*}
  Again by the central limit theorem, the probability on the
  right-hand side above tends to $0$ as $n \rightarrow \infty$,
  provided we choose $c_5 < c_4$.

  On the other hand, appealing one final time to the central
  limit theorem,
  \begin{equation*}
    \mu^+(\{ \sigma \st S(\sigma) > s^\star + c_5n^{-1/4} \}) =  o(1) .
  \end{equation*}
  Consequently,  
  \begin{align*}
    d_n( t_n^\star) 
    & \geq \P_{\one, \mu^+}
      \left(S_{t_n^\star} > s^\star + c_5 n^{-1/4}\right) \\
    & \quad - \mu^+(\{ \sigma \st S(\sigma) > s^\star +
      c_5n^{-1/4} \}) \\
    & = 1 - o(1),
  \end{align*}
  and so $\tmix(n) \geq (1/4)n\log n$ for $n$ large.
\end{proof}

\section{Conjectures} \label{sec.conj}

We believe the results proven in this paper should be generic
for Glauber dynamics on transitive graphs.

To be concrete, consider the $d$-dimensional torus
 $(\Z/n\Z)^d$.  Let $\beta_c$ be the critical temperature
for uniqueness of Gibbs measures on $\Z^d$. 

We make the following conjectures:
\begin{enumeratei}
  \item For $\beta < \beta_c$, there is a cut-off.
  \item For $\beta = \beta_c$, the mixing time
    is polynomial in $n$.  A stronger conjecture
    is that there is a critical dimension
    $d_c$ such that for $d \geq d_c$,
    the mixing time $\tmix$ is $O(|V_n|^{3/2})$.
  \item For $\beta > \beta_c$, if the dynamics
    are suitably truncated, the mixing time is
    polynomial in $n$.  A stronger version is that
    again there is a critical dimension
    $d_c$ such that for $d > d_c$, the mixing time is
    $O(|V_n|\log|V_n|)$. 
\end{enumeratei}

\section*{Acknowledgments}

This research was initiated during a visit by DAL and MJL to
Microsoft Research, and continued during both the DIMACS-Georgia Tech
Phase Transition workshop and the Park City Mathematics Institute.
Also, while working on this project, MJL was first supported in part by the
Nuffield Foundation, and then in part by the Humboldt Foundation.  We
thank Elchanan Mossel for useful discussions at an early stage of this
work, and Jian Ding and Eyal Lubetzky for comments on an early draft.


\begin{bibdiv}
\begin{biblist}
\bib{AH}{article}{
   author={Aizenman, M.},
   author={Holley, R.},
   title={Rapid convergence to equilibrium of stochastic Ising models in the
   Dobrushin Shlosman regime},
   conference={
      title={},
      address={Minneapolis, Minn.},
      date={1984--1985},
   },
   book={
      series={IMA Vol. Math. Appl.},
      volume={8},
      publisher={Springer},
      place={New York},
   },
   date={1987},
   pages={1--11},
}

\bib{AF:MC}{book}{
  author           = {Aldous, D.},
  author           = {Fill, J.},
  title            = {Reversible Markov chains and random walks on graphs},
  note             =
  {Manuscript available at
  \texttt{http://www.stat.berkeley.edu/\~{}aldous/RWG/book.html}},  
  status           = {in progress},
}

\bib{BLPW}{article}{
   author={Bender, E. A.},
   author={Lawler, G. F.},
   author={Pemantle, R.},
   author={Wilf, H. S.},
   title={Irreducible compositions and the first return to the origin of a
     random walk},
   journal={S\'em. Lothar. Combin.},
   volume={50},
   date={2003/04},
   pages={Art. B50h, 13 pp. (electronic)},
   eprint = {arxiv:math/0404253v1},
}	

\bib{BEGKa}{article}{
   author={Bovier, A.},
   author={Eckhoff, M.},
   author={Gayrard, V.},
   author={Klein, M.},
   title={Metastability in stochastic dynamics of disordered mean-field
   models},
   journal={Probab. Theory Related Fields},
   volume={119},
   date={2001},
   number={1},
   pages={99--161},

}
		
\bib{BEGK}{article}{
   author={Bovier, A.},
   author={Eckhoff, M.},
   author={Gayrard, V.},
   author={Klein, M.},
   title={Metastability and low lying spectra in reversible Markov chains},
   journal={Comm. Math. Phys.},
   volume={228},
   date={2002},
   number={2},
   pages={219--255},
}

\bib{BM:MetaStabGlaub}{article}{
   author={Bovier, A.},
   author={Manzo, F.},
   title={Metastability in Glauber dynamics in the low-temperature limit:
   beyond exponential asymptotics},
   journal={J. Statist. Phys.},
   volume={107},
   date={2002},
   number={3-4},
   pages={757\ndash 779},
}

\bib{BD:PC}{inproceedings}{
  author           = {Bubley, R.},
  author           = {Dyer, M.},
  title            = {Path coupling: A technique for proving rapid mixing
                      in Markov chains},
  booktitle        = {Proceedings of the 38th Annual Symposium on
                      Foundations of Computer Science},
  publisher        = {I.E.E.E.},
  address          = {Miami, FL},
  pages            = {223\ndash 231},
  year             = {1997},
}


\bib{D:COP}{article}{
   author={Diaconis, P.},
   title={The cutoff phenomenon in finite Markov chains},
   journal={Proc. Nat. Acad. Sci. U.S.A.},
   volume={93},
   date={1996},
   number={4},
   pages={1659--1664},
}

\bib{DSC}{article}{
   author={Diaconis, P.},
   author={Saloff-Coste, L.},
   title={Separation cut-offs for birth and death chains},
   journal={Ann. Appl. Probab.},
   volume={16},
   date={2006},
   number={4},
   pages={2098--2122},
}

\bib{ellis:eldsm}{book}{
   author={Ellis, R. S.},
   title={Entropy, large deviations, and statistical mechanics},
   series={Grundlehren der Mathematischen Wissenschaften},
   volume={271},
   publisher={Springer-Verlag},
   place={New York},
   date={1985},
}

\bib{ENR}{article}{
   author={Ellis, R. S.},
   author={Newman, C. M.},
   author={Rosen, J. S.},
   title={Limit theorems for sums of dependent random variables occurring in
   statistical mechanics. II. Conditioning, multiple phases, and
   metastability},
   journal={Z. Wahrsch. Verw. Gebiete},
   volume={51},
   date={1980},
   number={2},
   pages={153--169},
}		


\bib{F:v2}{book}{
  author = {Feller, W.},
  title  = {An Introduction to Probability Theory and its
  Applications},
  volume = {2},
  date   = {1971},
  edition = {second edition},
  publisher = {J. Wiley \& Sons},
  place = {New York},
}

\bib{GWL}{article}{
  journal = {Phys. Rev.},
  volume  = {149},
  pages   = {301 \ndash 305},
  year    = {1966},
  title   = {Relaxation Times for Metastable States
             in the Mean-Field Model of a Ferromagnet},
  author  = {Griffiths, R. B.},
  author  = {Weng, C.-Y.},
  author  = {Langer, J. S.},
}  

\bib{lpw}{book}{ 
	author = {Levin, D.},
	author = {Peres, Y.},
	author = {Wilmer, E.},
	title =  {Markov Chains and Mixing Times},
	year  =  {2007},
	note = {In preparation, available at 
	\texttt{http://www.uoregon.edu/\~{}dlevin/MARKOV/}},
}

\bib{L:CM}{book}{
   author={Lindvall, T.},
   title={Lectures on the coupling method},
   note={Corrected reprint of the 1992 original},
   publisher={Dover Publications Inc.},
   place={Mineola, NY},
   date={2002},
   pages={xiv+257},
}

\bib{OV}{book}{
   author={Olivieri, E.},
   author={Vares, M. E.},
   title={Large deviations and metastability},
   series={Encyclopedia of Mathematics and its Applications},
   volume={100},
   publisher={Cambridge University Press},
   place={Cambridge},
   date={2005},
   pages={xvi+512},
}

\bib{sg}{article}{
   author={Simon, B.},
   author={Griffiths, R. B.},
   title={The $(\phi \sp{4})\sb{2}$ field theory as a classical Ising model},
   journal={Comm. Math. Phys.},
   volume={33},
   date={1973},
   pages={145--164},
}

\bib{S:ARGC}{book}{
   author={Sinclair, A.},
   title={Algorithms for random generation and counting},
   series={Progress in Theoretical Computer Science},
   note={A Markov chain approach},
   publisher={Birkh\"auser Boston Inc.},
   place={Boston, MA},
   date={1993},
   pages={vi+146},
}

\end{biblist}
\end{bibdiv}
\end{document}